\documentclass[11pt]{amsart}
\usepackage{amssymb,mathrsfs,graphicx,enumerate}
\usepackage{amsmath,amsfonts,amssymb,amscd,amsthm,bbm}
\usepackage[retainorgcmds]{IEEEtrantools}
\usepackage{colortbl}
\usepackage{cite}

\title[Thermodynamic Kuramoto ensemble on a ring lattice]{Emergent behaviors of Thermodynamic Kuramoto ensemble on a regular ring lattice}

\author[Ha]{Seung-Yeal Ha}
\address[Seung-Yeal Ha]{\newline Department of Mathematical Sciences\newline Seoul National University, Seoul 08826, Republic of Korea, and \newline
Korea Institute for Advanced Study, Hoegiro 85, 02455, Seoul, Republic of Korea}
\email{syha@snu.ac.kr}

\author[Park]{Hansol Park}
\address[Hansol Park]{\newline Department of Mathematical Sciences\newline Seoul National University, Seoul 08826, Republic of Korea}
\email{hansol960612@snu.ac.kr}

\author[Ruggeri]{Tommaso Ruggeri}
\address[Tommaso Ruggeri]{\newline Department of Mathematics and
	Alma Mater Research Center
	on Applied Mathematics AM$^2$ \newline
	University of Bologna (Italy)}
\email{tommaso.ruggeri@unibo.it}

\author[Shim]{Woojoo Shim}
\address[Woojoo Shim]{\newline Department of Mathematical Sciences\newline Seoul National University, Seoul 08826, Republic of Korea}
\email{cosmo.shim@gmail.com }

\newtheorem{theorem}{Theorem}[section]
\newtheorem{lemma}{Lemma}[section]
\newtheorem{corollary}{Corollary}[section]
\newtheorem{proposition}{Proposition}[section]
\newtheorem{remark}{Remark}[section]

\newtheorem{definition}{Definition}[section]

\newcommand{\bbr}{\mathbb R}

\newcommand{\bbs}{\mathbb S}

\newcommand{\bx}{\mbox{\boldmath $x$}}

\newcommand{\bv}{\mbox{\boldmath $v$}}

\begin{document}

\date{\today}

\subjclass{82C10 82C22 35B37} \keywords{Emergence, entropy principle, Kuramoto model, thermodynamics}

\thanks{\textbf{Acknowledgment.} The work of S.-Y. Ha was supported by National Research Foundation of Korea(NRF-2017R1A2B2001864), and 
the work of T. Ruggeri was supported National Group of Mathematical Physics GNFM-INdAM}
\begin{abstract}
The temporal evolution of Kuramoto oscillators influenced by the temperature field often appears in biological oscillator ensembles. In this paper, we propose a generalized Kuramoto type lattice model on a regular ring lattice with the equal spacing assuming that each oscillator has an internal energy (temperature). Our lattice model is derived from the thermodynamical Cucker-Smale model for flocking on the 2D free space under the assumption that the ratio between velocity field and temperature field at each lattice point has a uniform magnitude over lattice points. The proposed model satisfies an entropy principle and exhibits emergent dynamics under some sufficient frameworks formulated in terms of initial data and system parameters. Moreover, the phase-field tends to the Kuramoto phase-field asymptotically. 
\end{abstract}

\maketitle \centerline{\date}


\section{Introduction} \label{sec:1}
\setcounter{equation}{0}
Synchronization of phase-coupled oscillators often appears in our biological and chemical complex oscillatory systems \cite{A-B, B-B, D-B1, Pe, P-R, Wi1}. In spite of its ubiquity in our nature, a rigorous mathematical study based on mathematical models has been initiated by two pioneers Arthur Winfree \cite{Wi2} and Yoshiki Kuramoto \cite{Ku2} about a half-century ago. In this paper, we present a generalized Kuramoto lattice model on a ring lattice for the phase evolution of Kuramoto oscillators assuming that each oscillator has its own internal temperature i.e.,  a coupled phase-temperature model on the ring lattice. For the mathematical modeling for such a situation, we consider a finite size ensemble of Kuramoto oscillators (or rotators) on the unit circle under the influence of a temperature field. Interactions between Kuramoto oscillators and temperature fields were already discussed in biology literature, e.g., \cite{N-G-R-P, R-L-D, Z-V} without any explicit mathematical models. We can reasonably argue that the synchronizability of Kuramoto oscillators will be inversely proportional to the size of temperature, i.e., as the temperature becomes higher, the synchronizability of oscillator ensemble will be lowered. However, as far as authors know, this plausible interaction scenario between the phase-field and temperature field has not been addressed in mathematics and physics literature via oscillator models and their rigorous analysis. In this paper, we are mainly interested in the mathematical modeling and the rigorous justification of our plausible interaction scenario. To motivate our discussion, we begin with the Kuramoto model whose emergent dynamics is extensively studied in literature, to name a few, \cite{A-R, B-C-M, C-H-J-K, C-S, Da, D-X, D-B, H-K-L, H-K-L1, H-K-R, H-L-X, J-M-B, L-H, M-S3, M-S2, M-S1,S-K, S-W, St, V-M1, V-M2}. 

Consider a ring lattice on $\bbs^1$ consisting of $N$ lattice points with heterogeneous coupling weights $\{\psi_{\alpha \beta} \}$. Let
 $\theta_\alpha = \theta_\alpha(t)$ and $\nu_\alpha$ be the Kuramoto phase and natural frequency evaluated at the $\alpha$-th lattice point $x_\alpha \in \bbs^1$, respectively. Then, the temporal evolution of $\theta_\alpha$ is given by the following first-order system:
\begin{equation} \label{KM}
{\dot \theta}_\alpha = \nu_\alpha + \frac{\kappa}{N} \sum_{\beta=1}^{N} \psi_{\alpha \beta} \sin(\theta_\beta - \theta_\alpha), \quad \alpha = 1, \cdots, N,
\end{equation}
where $\kappa$ is a nonnegative coupling strength, and positive constant $\psi_{\alpha \beta}$ takes into account the coupling strength between oscillators which may depend on the mutual distance between lattice points. For simplicity, we assume that the lattice points are fixed and therefore weights are constants (see \cite{H-L-X,H-J-K} for example). Note that the Kuramoto model \eqref{KM} is a kind of {\it mechanical model} that does not take into account of temperature field. Then, a natural question is 
\begin{quote}
``How to model the phase evolution of a Kuramoto ensemble in a self-consistent temperature field?"
\end{quote}
i.e., how to couple a temperature field to \eqref{KM} so that the resulting model exhibits synchronization under certain conditions, and at the same time, it should be consistent with the entropy principle. In order to introduce a reasonable thermodynamic Kuramoto (TK) model, we use Kuramoto's heuristic argument which was used for the derivation of the Kuramoto model \eqref{KM} from the Landau-Stuart model, and relation between the Cucker-Smale model and Kuramoto model \cite{H-J-K}. More precisely, we apply the above outlined argument to the thermodynamic Cucker-Smale(TCS) model to derive the thermodynamic Kuramoto model on a regular ring lattice with the same spacing $\{x_\alpha \} \subset \bbs^1$. Let $T_\alpha = T_\alpha(t)$ be the scalar-valued temperature at the lattice point $x_\alpha$ and at time $t$. Then, we assume that the dynamics of $\theta_\alpha$ and $T_\alpha$ are governed by the following Cauchy problem to the thermodynamic Kuramoto(TK) model:
\begin{equation} \label{TKM}
\begin{cases}
\displaystyle \dot{\theta}_\alpha = \nu_\alpha+\displaystyle\frac{\kappa_1}{N}\sum_{\beta=1}^N \frac{\psi_{\alpha\beta}}{T_\alpha} \sin(\theta_\beta-\theta_\alpha),~~t > 0, \\
\displaystyle \dot{T}_\alpha = \frac{\kappa_2}{N}\sum_{\beta=1}^N \frac{\zeta_{\alpha\beta}T_*^2}{(T_*^2+\eta^2T_\alpha)}\left(\frac{1}{T_\alpha}-\frac{1}{T_\beta}\right), \\
\displaystyle (\theta_\alpha, T_\alpha) \Big|_{t = 0} = (\theta_\alpha^{in}, T_\alpha^{in}), \quad \alpha = 1, \cdots, N.
\end{cases}
\end{equation}

Note that for the isothermal case with the same temperatures $T_\alpha = T_0$ for all $\alpha$, the TK model reduces to the Kuramoto model \eqref{KM} with $\kappa = \frac{\kappa_1}{T_0}$. In fact, we will show that the temperature field will approach to the constant value asymptotically so that the Kuramoto model will govern the phase dynamics for \eqref{TKM} asymptotically. As long as the initial temperatures are strictly positive, the temperature field takes positive values (see Lemma \ref{L2.1}). Hence, local well-posedness of the TK model will follow from the standard Cauchy-Lipschitz theory as long as the initial temperatures are strictly positive. Hence, the interesting question for \eqref{TKM} lies on the large-time behavior of \eqref{TKM}.  \newline

The main results of this paper are three-fold. First, we provide a heuristic derivation of the TK model from the TCS model by adapting Kuramoto's formal arguments (see Section \ref{sec:2.2}). Second, we provide several sufficient frameworks leading to the emergent dynamics of the TK model. For a homogeneous ensemble with the same natural frequency confined in a half circle, we show that the complete phase synchronization emerges exponentially fast, as long as the initial temperatures and coupling strengths are strictly positive, we show that there exists a constant asymptotic state $(\theta^{\infty}, T^{\infty})$ such that 
\[  \lim_{t \to \infty} \max_{1 \leq \alpha \leq N} \Big ( |\theta_\alpha(t) - \theta^{\infty} | + |T_\alpha(t) - T^{\infty} | \Big) = 0,\]
(see Theorem \ref{T3.0} and Theorem \ref{T3.1} for details). On the other hand, for a heterogeneous ensemble with distributed natural frequencies, we show that for some specific class of initial phase configuration, phase-diameter is uniformly bounded and asymptotically controlled by the quantity of the order of $ \kappa_1^{-1}$ and $T$: for the temperature limit $T^\infty$ and natural frequency diameter $D(\nu):=\max_{\alpha, \beta}|\nu_\alpha-\nu_\beta|$, we have
\[ \limsup_{t \to \infty} {\mathcal D}(\Theta(t)) \leq \arcsin\left(\frac{D(\nu)T^\infty}{\kappa_1\psi_{\min}}\right),\quad \psi_{\min}:=\min_{1 \leq \alpha, \beta \leq N}\psi_{\alpha \beta}>0, \]
(see Proposition \ref{P3.2}). Moreover, we also provide a sufficient framework leading to the convergence of phase-locking: there exists asymptotic states $\Theta^{\infty} = (\theta_1^{\infty}, \cdots, \theta_N^{\infty})$ and $T^\infty$ such that 
\[  \lim_{t \to \infty} \max_{1 \leq \alpha \leq N} \Big ( |\theta_\alpha(t) - \theta_\alpha^{\infty} | + |T_\alpha(t) - T^{\infty} | \Big) = 0. \]
Third, we show that the asymptotic dynamics of the TK model can be approximately governed by the Kuramoto model (see Theorem \ref{T4.1}).  \newline

The rest of this paper is organized as follows. In Section \ref{sec:2}, we briefly discuss the thermodynamic Cucker-Smale model and its basic properties, and then provide a heuristic derivation of the thermodynamic Kuramoto model with a priori estimates on temperature. In Section \ref{sec:3}, we provide several emergent dynamics of the TK model for homogeneous and heterogeneous ensembles. In Section \ref{sec:4}, we study the asymptotic equivalence between the TK model and the Kuramoto model. Finally, Section \ref{sec:5} is devoted to a brief summary of our main results. \newline

\noindent {\bf Notation}: The constant $C$ denotes a generic positive constant which may differ from line to line, and the relation $A(t)  \lesssim B(t)$ represents an inequality
$A(t) \leq C B(t)$ for a positive constant $C$ and all $t\geq 0$. For notational simplicity, we also use handy notation from time to time:
\[ \max_{\alpha}  := \max_{1 \leq \alpha \leq N}, \quad  \max_{\alpha, \beta}  := \max_{1 \leq \alpha, \beta \leq N}, \quad \sum_{\alpha} :=  \sum_{\alpha =1}^{N}, \quad \sum_{\alpha, \beta} :=  \sum_{\alpha = 1}^{N} \sum_{\beta= 1}^{N}. \]
Also, for $N$-dimensional vector $X=(x_1,\cdots,x_N)$, we denote by $\mathcal{D}(X)$ as a maximal difference between $x_\alpha'$s:
\[\mathcal{D}(X):=\max_{\alpha, \beta}|x_\alpha-x_\beta|. \]

\section{Preliminaries} \label{sec:2}
\setcounter{equation}{0}
In this section, we first present the thermodynamic Cucker-Smale(TCS) model for the flocking dynamics of a Cucker-Smale ensemble in a self-consistent temperature field, and study its basic properties such as Galilean invariance, conservation laws, and entropy principle. Then, we derive a thermodynamic Kuramoto(TK) model by applying Kuramoto's heuristic argument to the TCS model. Our derived phase-temperature coupled model will inherit entropy principle from the TCS model, but lose conservation law for the total phase in the derivation procedure.
\subsection{The TCS model} \label{sec:2.1} 
Ha and Ruggeri \cite{H-R}  observed an analogy between the extended theory of mixture of gases with multi-temperatures and the
CS flocking model. Under the spatial homogeneity, a hyperbolic PDE model for a mixture of gases \cite{R-Sim,R-S} reduces to a particle model which includes the CS model as a special case. For this model, the authors coined the name of thermomechanical Cucker-Smale (TCS) model for a generalized CS model with temperature as an internal variable. For the isothermal case, this model reduces to the CS flocking model.

Consider an ensemble of Cucker-Smale flocking particles in a self-consistent temperature field. Let $\bx_\alpha, ~\bv_\alpha$ and $T_\alpha$  be the position, velocity and temperature of the $\alpha$-th Cucker-Smale particle, respectively. Then, the dynamics of TCS particle is given by the following first-order system for $(\bx_\alpha, \bv_\alpha, T_\alpha)$:
\begin{equation}\label{TCS}
\begin{cases}
\displaystyle \frac{d \bx_\alpha}{dt} =\bv_\alpha, \quad t > 0,~~\alpha = 1, \cdots, N, \\
\displaystyle \frac{d \bv_\alpha}{dt} = \frac{\kappa_1}{N}\sum_{\beta}\psi_{\alpha\beta}\left(\frac{\bv_\beta-\bv}{T_\beta}-\frac{\bv_\alpha - \bv}{T_{\alpha}}\right),\\
\displaystyle \frac{d}{dt} \Big(T_\alpha + \frac{1}{2} |\bv_\alpha|^2 \Big) = \frac{\kappa_2}{N}\sum_{\beta}\zeta_{\alpha\beta}\left(\frac{1}{T_\alpha}-\frac{1}{T_\beta}\right) \\
\displaystyle \hspace{3.5cm} + \frac{\kappa_1}{N}\sum_{\beta}\psi_{\alpha\beta}\left(\frac{\bv_\beta-\bv}{T_\beta}-\frac{\bv_\alpha - \bv}{T_{\alpha}}\right) \cdot \bv,
\end{cases}
\end{equation}
where $\bv:= \frac{1}{N} \sum_{\beta} \bv_\beta$ is the average velocity, and $\kappa_1, \kappa_2$ are positive coupling strengths.  The second equation \eqref{TCS}$_2$ is the momentum balance  equation  of each component in  the case of  a mixture of gas, while the third equation  \eqref{TCS}$_3$ is the balance of energy of each species. For simplicity, we take the common specific heat and unit density, and we also assumed that internal energy of each species is  linear in the temperature.  Therefore, individual energy density consisting of internal energy and kinetic energy becomes $T_\alpha + \frac{1}{2} |\bv_\alpha|^2$. For the mechanical case with common constant temperature, system \eqref{TCS} reduces to the first two equations and coincides with the classical Cucker-Smale model. We observe that the case in which the matrices $(\psi_{\alpha \beta})$ and $(\zeta_{\alpha \beta})$ depending of mutual distance of particles was also considered in \cite{H-K-Rug} and the relativistic model of TCS was recently proposed by Ha-Kim-Ruggeri  \cite{H-K-Rug2}.
\bigskip

Next, we recall that the TCS model satisfies Galilean invariance and global conservation laws.
\begin{proposition} \label{P2.1}
Let $\{ (\bx_\alpha, \bv_\alpha, T_\alpha) \}$ be a solution to \eqref{TCS}. Then, the following two assertions hold.
\begin{enumerate}
\item
(Galilean invariance): For any constant vector ${\mathbf c} \in \bbr^d$, system \eqref{TCS} is invariant under Galilean transformation:
\[ (\bx_\alpha, \bv_\alpha, T_\alpha) \quad \Longrightarrow \quad (\bx_\alpha + t {\mathbf c}, \bv_\alpha + {\mathbf c}, T_\alpha). \]
\item
(Global conservation laws): The total momentum and energy are conserved along the flow \eqref{TCS}:
\[ \frac{d}{dt} \sum_{\alpha} \bv_\alpha = 0 \quad \mbox{and} \quad \frac{d}{dt} \sum_{\alpha} \Big(T_\alpha + \frac{1}{2} |\bv_\alpha|^2 \Big) =0, \quad t > 0.\]
\end{enumerate}
\end{proposition}
\begin{proof} (i) The Galilean invariance can be easily seen from the following observations:
\[ (\bv_\alpha + {\mathbf c}) - \frac{1}{N} \sum_{\alpha} ((\bv_\alpha + {\mathbf c}) ) = \bv_\alpha - \frac{1}{N} \sum_{\alpha} \bv_\alpha = \bv_\alpha  - \bv,\]
and 
\[  \frac{d}{dt} \Big(T_\alpha + \frac{1}{2} |\bv_\alpha + {\mathbf c}|^2 \Big) =   \frac{d}{dt} \Big(T_\alpha + \frac{1}{2} |\bv_\alpha|^2 \Big)   + {\mathbf c} \cdot \frac{d \bv_\alpha}{dt}.   \]

\vspace{0.2cm}

\noindent (ii) The conservation of momentum and energy follow from the anti-symmetry of the R.H.S. in $\eqref{TCS}_2$  and $\eqref{TCS}_3$ by the index transformation $(\alpha, \beta)~\longleftrightarrow~(\beta, \alpha)$.  
\end{proof}
Since the total momentum is conserved, ${\bf v}$ is constant  and from Galilean invariance, without loss of generality, we can take the rest frame for which
\[  \bv = 0.  \]
In this  frame, the TCS model takes a simpler form:
\begin{equation}\label{TCS-1}
\begin{cases}
\displaystyle \frac{d \bx_\alpha}{dt} =\bv_\alpha,  \quad t > 0,~~\alpha = 1, \cdots, N, \\
\displaystyle \frac{d \bv_\alpha}{dt} = \frac{\kappa_1}{N}\sum_{\beta}\psi_{\alpha\beta}\left(\frac{\bv_\beta}{T_\beta}-\frac{\bv_\alpha}{T_{\alpha}}\right),\\
\displaystyle \frac{d}{dt} \Big(T_\alpha + \frac{1}{2} |\bv_\alpha|^2 \Big) = \frac{\kappa_2}{N}\sum_{\beta}\zeta_{\alpha\beta}\left(\frac{1}{T_\alpha}-\frac{1}{T_\beta}\right),
\end{cases}
\end{equation}
and now we introduce an entropy functional ${\mathcal S}$ for \eqref{TCS-1}:
\begin{equation} \label{Entropy}
\mathcal{S}(t) := \sum_{\alpha} \ln T_\alpha(t).
\end{equation}
\begin{proposition} \label{P2.2}
Let $\{ (\bx_\alpha, \bv_\alpha, T_\alpha) \}$ be a solution to \eqref{TCS}. Then, the total entropy is non-decreasing:
\[ \frac{d{\mathcal S}}{dt} \geq 0, \quad t > 0. \]
\end{proposition}
\begin{proof}
The proof can be found in \cite{H-R}.
\end{proof}
 \subsection{The TK model} \label{sec:2.2}
 In this subsection, we present the derivation of the TK model from the TCS model by adapting heuristic argument which has been used in the derivation of the Kuramoto model from the Landau-Stuart model, and then we study several a priori estimates on temperatures for later discussion.
 
 \subsubsection{Derivation of the TK model} \label{sec:2.2.1}
Consider an ensemble of TCS particles on the plane, and we follow the Kuramoto's heuristics by assuming  the following relation: 
 \[  \Big| \frac{\bv_\alpha}{T_\alpha} \Big| = \Big| \frac{\bv_\beta}{T_\beta} \Big|, \quad \alpha, \beta = 1, \cdots, N. \]
 Now, we introduce a heading angle (or phase) $\theta_\alpha$ such that
 \begin{equation} \label{B-4}
 \frac{\bv_\alpha}{T_\alpha} = \frac{\eta}{T_*} e^{{\mathrm i} \theta_\alpha},  \quad \mbox{or equivalently} \quad \bv_\alpha = \eta \frac{T_\alpha}{T_*} e^{{\mathrm i} {\theta_\alpha} }, \quad \alpha = 1, \cdots, N,
 \end{equation}
 where $\eta$ is a constant parameter with speed dimension, and $T_*$ is a constant reference temperature. \newline
 
\noindent $\bullet$~(Derivation of phase evolution):  We substitute the ansatz \eqref{B-4} into $\eqref{TCS-1}_2$ and $\eqref{TCS-1}_3$ to see
\[
\dot{\bv}_\alpha =\frac{\eta}{T_*}\dot{T}_\alpha e^{\mathrm{i}\theta_\alpha}+\frac{\eta}{T_*} T_\alpha e^{\mathrm{i}\theta_\alpha} \mathrm{i}\dot{\theta}_\alpha 
= \frac{\kappa_1}{N}\sum_{\beta}\psi_{\alpha\beta}\left(\frac{\eta}{T_*} e^{\mathrm{i}\theta_\beta}-\frac{\eta}{T_*} e^{\mathrm{i}\theta_\alpha}\right),
\]
or after further simplification,
\begin{align*}
\dot{T}_\alpha +   \mathrm{i} T_\alpha \dot{\theta}_\alpha = \frac{\kappa_1}{N}\sum_{\beta}\psi_{\alpha\beta}\left( e^{\mathrm{i}(\theta_{\beta}-\theta_\alpha)}-1 \right).
\end{align*}
Now, we compare the imaginary part of the above relation to get 
\[ T_\alpha\dot{\theta}_\alpha= \frac{\kappa_1}{N}\sum_{\beta}\psi_{\alpha\beta}\sin(\theta_\beta-\theta_\alpha). \]
Once all temperatures $T_\alpha$ are strictly positive, then the above relation can be rewritten as 
\begin{equation} \label{B-5}
{\dot \theta}_\alpha =  \frac{\kappa_1}{N T_\alpha}\sum_{\beta}\psi_{\alpha\beta}\sin(\theta_\beta-\theta_\alpha),
\end{equation}
and we may add a natural frequency $\nu_\alpha$ in the R.H.S of \eqref{B-5} to obtain a Kuramoto-type model for $\theta_\alpha$:
\begin{equation} \label{B-5-1}
{\dot \theta}_\alpha =  \nu_\alpha + \frac{\kappa_1}{N T_\alpha}\sum_{\beta}\psi_{\alpha\beta}\sin(\theta_\beta-\theta_\alpha).
\end{equation}

\vspace{0.2cm}

\noindent $\bullet$~(Derivation of temperature evolution): Note that relation \eqref{B-4} implies
\begin{equation} \label{B-6}
\frac{1}{2} |\bv_\alpha|^2 = \frac{\eta^2}{2T_*^2} |T_\alpha|^2, \quad \mbox{i.e.,} \quad \frac{d}{dt}  |\bv_\alpha|^2 = \frac{2\eta^2}{T_*^2} T_\alpha {\dot T}_\alpha. 
\end{equation}
Hence, the equation $\eqref{TCS-1}_3$ becomes 
\[    {\dot T}_\alpha + \frac{\eta^2}{T_*^2} T_\alpha {\dot T}_\alpha = \frac{\kappa_2}{N}\sum_{\beta}\zeta_{\alpha\beta}\left(\frac{1}{T_\alpha}-\frac{1}{T_\beta}\right),            \]
or equivalently, 
\begin{equation} \label{B-7}
{\dot T}_\alpha = \frac{\kappa_2T_*^2}{N(T_*^2+\eta^2T_\alpha)}\sum_{\beta}\zeta_{\alpha\beta}\left(\frac{1}{T_\alpha}-\frac{1}{T_\beta}\right).
\end{equation}
Finally, our desired Kuramoto-type model consists of two equations \eqref{B-5-1} and \eqref{B-7}:
\begin{equation} \label{B-8}
\begin{cases}
\displaystyle \dot{\theta}_\alpha = \nu_\alpha+\displaystyle\frac{\kappa_1}{N}\sum_{\beta} \frac{\psi_{\alpha\beta}}{T_\alpha} \sin(\theta_\beta-\theta_\alpha), \\
\displaystyle \dot{T}_\alpha = \frac{\kappa_2}{N}\sum_{\beta} \frac{\zeta_{\alpha\beta}T_*^2}{(T_*^2+\eta^2T_\alpha)}\left(\frac{1}{T_\alpha}-\frac{1}{T_\beta}\right),
\end{cases}
\end{equation}
where network topologies $(\psi_{\alpha \beta})$ and $(\zeta_{\alpha \beta})$ are assumed to be positive and symmetric:
\begin{equation} \label{B-9}
\psi_{\alpha \beta}=\psi_{\beta\alpha} > 0, \qquad \zeta_{\alpha \beta}=\zeta_{\beta\alpha} > 0, \quad 1 \leq \alpha, \beta \leq N.
\end{equation}
From now on, we call system \eqref{B-8} as the {\it ``thermodynamic Kuramoto(TK) model"} throughout the paper.

As a thermodynamically consistent model, we here consider the entropy $\mathcal{S}$ in \eqref{Entropy} as in TCS model.
Then, Lemma \ref{L2.1} immediately guarantees that the function ${\mathcal S}$ is well-defined as long as the initial temperatures are all strictly positive, and we can verify the following entropy principle consistent with the second law of thermodynamics:
\begin{proposition} \label{P2.3}
    Let $\{T_\alpha \}_{\alpha=1}^N$ be a solution to $\eqref{B-8}_2$, where the initial data $\{T_\alpha^{in} \}_{\alpha=1}^N$ and system parameters $\kappa_2,\zeta_{\alpha \beta}$ satisfy 
    \[T_\alpha^{in}>0,\quad \kappa_2\geq 0,\quad \zeta_{\alpha \beta}\geq 0.\quad \forall~ \alpha,\beta=1,\cdots,N. \] Then, the total entropy is nondecreasing over time: 
    \[\frac{d\mathcal{S}}{dt}\geq 0,\quad t>0. \] 
\end{proposition}    
    \begin{proof}
        Note that the temporal derivative of $\mathcal{S}$ has a form 
        \[ 
        \frac{d\mathcal{S}}{dt} =\sum_{\alpha}\frac{1}{T_\alpha}\frac{dT_\alpha}{dt} =\frac{\kappa_2}{N}\sum_{\alpha,\beta}\frac{\zeta_{\alpha \beta}T_*^2}{T_\alpha(T_*^2+\eta^2T_\alpha) }\left(\frac{1}{T_\alpha}-\frac{1}{T_\beta} \right).
        \]
        Then, we use the index changing for $\alpha$ and $\beta$ in the above summation to conclude the non-decreasing of $\mathcal{S}$:
        \[\begin{aligned}
        \frac{d\mathcal{S}}{dt}&=\frac{\kappa_2}{N}\sum_{\alpha,\beta}\frac{\zeta_{\alpha \beta}T_*^2}{T_\alpha(T_*^2+\eta^2T_\alpha) }\left(\frac{1}{T_\alpha}-\frac{1}{T_\beta} \right)\\
        &=\frac{\kappa_2}{2N}\sum_{\alpha,\beta}\left(\frac{\zeta_{\alpha \beta}T_*^2}{T_\alpha(T_*^2+\eta^2T_\alpha) }-\frac{\zeta_{\beta \alpha}T_*^2}{T_\beta(T_*^2+\eta^2T_\beta) }\right)\left(\frac{1}{T_\alpha}-\frac{1}{T_\beta}\right)\\
        &=\frac{\kappa_2}{2N}\sum_{\alpha,\beta}\frac{\zeta_{\alpha \beta}T_*^2\left[T_\beta(T_*^2+\eta^2T_\beta)-T_\alpha(T_*^2+\eta^2T_\alpha)\right](T_\beta-T_\alpha) }{T_\alpha^2T_\beta^2(T_*^2+\eta^2T_\alpha)(T_*^2+\eta^2T_\beta)}\\
        &=\frac{\kappa_2}{2N}\sum_{\alpha,\beta}\frac{\zeta_{\alpha \beta}T_*^2\left[T_*^2+\eta^2(T_\alpha+T_\beta) \right](T_\beta-T_\alpha)^2 }{T_\alpha^2T_\beta^2(T_*^2+\eta^2T_\alpha)(T_*^2+\eta^2T_\beta)}\\
        &\geq 0.
        \end{aligned} \]
    \end{proof}

\subsubsection{Temperature relaxation} \label{sec:2.2.2}
In this part, we study positivity and exponential alignment for temperature field and asymptotic conservation law for the TK model. In the next lemma, we show that $T_\alpha$ are bounded above and has a positive lower bound along with the thermodynamic Kuramoto flow. This monotonicity result on maximal and minimal temperatures guarantees the global well-posedness of \eqref{B-8} -- \eqref{B-9}.

\begin{lemma} \label{L2.1}
    Let $\{(\theta_{\alpha}, T_{\alpha})\}_{\alpha=1}^N$ be a solution to  system \eqref{B-8} -- \eqref{B-9} in a finite time-interval $[0, \tau)$, where the initial data $\{(\theta^{in}_{\alpha}, T^{in}_{\alpha})\}_{\alpha=1}^N$ and system parameters $\kappa_1,\kappa_2$ satisfy
    \[  \kappa_1,\kappa_2> 0,\quad \min_{\alpha} T^{in}_\alpha > 0.   \]
     Then, we have
    \[ 0<\min_{\alpha} T^{in}_\alpha\leq  \min_{\alpha} T_{\alpha}(t) \leq \max_{\alpha} T_{\alpha}(t)\leq \max_{\alpha} T^{in}_{\alpha},\quad t\in(0,\tau).   \]
\end{lemma}
\begin{proof}   Since the R.H.S. of \eqref{B-8} is analytic for $\theta_\alpha\in\mathbb{R}$ and $T_\alpha>0$, the solution $(\theta_\alpha, T_\alpha)$ is real analytic until some $T_\alpha$ becomes zero. In particular, the zero set of $T_\alpha - T_\beta$ is finite in any finite-time interval. Hence, there are only finitely many crossings between $T_\alpha'$s, i.e.,  there exists a time sequence $\{\tau_i\}_{i=0}^n$ with order 
        \[ 0 = \tau_0 < \tau_1 < \cdots < \tau_n = \tau, \]
        such that we can choose a unique minimal index $m$ on each time-interval $(\tau_{i-1}, \tau_i)$:
        \[ T_m(t) = \min_{\alpha} T_\alpha(t), \quad t \in (\tau_{i-1}, \tau_i). \]
        Moreover, on the interval $(\tau_{i-1}, \tau_i)$, one has
        \[  \dot{T}_m = \frac{\kappa_2}{N}\sum_{\beta} \frac{\zeta_{m\beta}T_*^2}{(T_*^2+\eta^2T_m)}\left(\frac{1}{T_m}-\frac{1}{T_\beta}\right) \geq 0.    \]  Then, we combine the above monotonic increasing property in each time interval $(\tau_{i-1},\tau_i)$ and the continuity of the minimal temperature $T_m$ to conclude the desired result. We can also use the same argument for the maximal temperature $\max_{\alpha} T_\alpha$ to show the monotonic decreasing property.
\end{proof}
Note that temperature field $\{T_\alpha\}_\alpha$ has a decoupled dynamics from that of phase, thus we can further analyze the asymptotic behavior. First, we set 
\[ \zeta_{\min}:=\min_{\alpha,\beta} \zeta_{\alpha \beta}, \quad \zeta_{\max}:=\max_{\alpha,\beta} \zeta_{\alpha \beta}. \]
\begin{proposition}\label{P2.4}
    Let $\{(\theta_{\alpha},T_\alpha) \}_{\alpha=1}^N$ be a global-in-time solution to the system \eqref{B-8} -- \eqref{B-9}, where the initial data $\{(\theta^{in}_{\alpha}, T^{in}_{\alpha})\}_{\alpha=1}^N$ and system parameters $\kappa_1,\kappa_2$ satisfy
    \[\kappa_1,\kappa_2> 0,\quad 0<T_1^{in} \leq \cdots\leq T_N^{in} <\infty.  \]
    Then, there exists a positive asymptotic temperature $T^\infty>0$ determined by the relation
    \begin{equation} \label{B-10}
    T^\infty+\frac{\eta^2}{2T_*^2}(T^{\infty})^2 =\frac{1}{N}\sum_{\alpha} \left({T^{in}_\alpha}+\frac{\eta^2}{2T_*^2}{T^{in}_\alpha}^2\right), 
    \end{equation}
    and we have an exponential consensus of temperatures to $T^\infty$: for every positive $\Lambda < \frac{\kappa_2\zeta_{\min}T_*^2}{{{T^\infty}^2}(T_*^2+\eta^2{T^\infty})}$, we have
    \[ \max_{\alpha}|T_\alpha(t)-T^\infty| \lesssim e^{-\Lambda t}. \]

\end{proposition}
\begin{proof}
    We first prove the existence of the common limit $T^\infty$. To prove this, we estimate the difference between maximal and minimal temperatures for each time $t>0$. Since Lemma \ref{L2.1} guarantees the global well-posedness of \eqref{B-8}--\eqref{B-9}, there exists a discrete time sequence $\{\tau_i\}_{i\geq 0}$  with order 
    \[0=\tau_0<\tau_1<\cdots, \]
    such that we can choose a unique maximal and minimal indices $M$ and $m$ for temperatures $\{T_\alpha\}_{\alpha}$ on each time interval $(\tau_{i-1},\tau_i)$:
    \[T_M(t):=\max_{\alpha}T_\alpha(t),\quad T_m(t):=\min_{\alpha}T_\alpha(t),\quad t\in(\tau_{i-1},\tau_i). \]
    Then, we further analyze the maximal difference $T_M-T_m$ among temperatures 
    \[\begin{aligned}
    \dot{T}_M-\dot{T}_m&=\frac{\kappa_2}{N}\sum_{\beta} \frac{\zeta_{M\beta}T_*^2}{(T_*^2+\eta^2T_M)}\left(\frac{1}{T_M}-\frac{1}{T_\beta}\right)-\frac{\kappa_2}{N}\sum_{\beta} \frac{\zeta_{m\beta}T_*^2}{(T_*^2+\eta^2T_m)}\left(\frac{1}{T_m}-\frac{1}{T_\beta}\right)\\
    &= \frac{\kappa_2}{N}\sum_{\beta} \frac{\zeta_{M\beta}T_*^2}{(T_*^2+\eta^2T_M)}\left(\frac{T_\beta-T_M}{T_MT_\beta}\right)+\frac{\kappa_2}{N}\sum_{\beta} \frac{\zeta_{m\beta}T_*^2}{(T_*^2+\eta^2T_m)}\left(\frac{T_m-T_\beta}{T_mT_\beta}\right)\\
    &\leq \frac{\kappa_2}{N}\sum_{\beta} \frac{\zeta_{M\beta}T_*^2}{(T_*^2+\eta^2T_M)}\left(\frac{T_\beta-T_M}{T_M^2}\right)+\frac{\kappa_2}{N}\sum_{\beta} \frac{\zeta_{m\beta}T_*^2}{(T_*^2+\eta^2T_M)}\left(\frac{T_m-T_\beta}{T_M^2}\right),
    \end{aligned} \]
    and obtain the following Gr\"onwall inequality by using Lemma \ref{L2.1}:
    \[\begin{aligned}
    \frac{d}{dt}({T}_M-{T}_m)\leq -\frac{\kappa_2\zeta_{\min}T_*^2}{{T_M}^2(T_*^2+\eta^2{T_M})}(T_M-T_m)\leq -\frac{\kappa_2\zeta_{\min}T_*^2}{{T^{in}_N}^2(T_*^2+\eta^2{T^{in}_N})}(T_M-T_m),
    \end{aligned} \]
    which gives the existence of the common limit $T^\infty$. Moreover, when a real number $\Lambda < \frac{\kappa_2\zeta_{\min}T_*^2}{{{T^\infty}^2}(T_*^2+\eta^2{T^\infty})}$ is given, there exists a time $t_\Lambda>0$ such that 
    \[ \Lambda < \frac{\kappa_2\zeta_{\min}T_*^2}{{T_M^2(t_\Lambda)}(T_*^2+\eta^2{T_M(t_\Lambda)})}.\]
    Therefore, we use the Gr\"onwall inequality again starting from $t=t_\Lambda$ and obtain
    \[|T_\alpha(t)-T^\infty|\leq |T_M(t)-T_m(t)|\leq |T_M(t_\Lambda)-T_m(t_\Lambda)|e^{-\Lambda(t-t_\Lambda)},\quad t>t_\Lambda,~ \alpha=1,\cdots,N. \]
    Finally, the common limit $T^\infty$ can be determined by the following conservation law:
    \[\frac{d}{dt}\sum_{\alpha} \left(T_\alpha+\frac{\eta^2}{2T_*^2}T_\alpha^2 \right)=\frac{\kappa_2}{N}\sum_{\alpha, \beta}\zeta_{\alpha \beta}\left(\frac{1}{T_\alpha}-\frac{1}{T_\beta} \right)=0. \]
\end{proof}
As a corollary of Proposition \ref{P2.4}, we have the asymptotic conservation law for the average of the total phase. 
\begin{corollary} \label{C2.1}
Let $\{ (\theta_{\alpha},T_\alpha) \}_{\alpha=1}^N$ be a global solution to system \eqref{B-8} -- \eqref{B-9}, where the initial data $\{(\theta^{in}_{\alpha}, T^{in}_{\alpha})\}_{\alpha=1}^N$ and system parameters $\kappa_1,\kappa_2$ satisfy
\[\kappa_1,\kappa_2> 0,\quad 0<T_1^{in} \leq \cdots\leq T_N^{in} <\infty.  \] Then, one has
\[  \Big|  \frac{d}{dt} \Big[  \sum_{\alpha} \theta_\alpha - t \sum_{\alpha} \nu_\alpha \Big ]  \Big|   \lesssim e^{-\Lambda t} \quad \mbox{as $t \to \infty$},     \]
i.e. the quantity $ \sum_{\alpha} \theta_\alpha - t \sum_{\alpha} \nu_\alpha$ converges exponentially.
\end{corollary}
\begin{proof}
We add $\eqref{B-8}_1$ over $\alpha$, and use the symmetry of $\psi_{\alpha \beta}$ to get 
\begin{align*}
\begin{aligned}
\sum_{\alpha} \dot{\theta}_\alpha - \sum_{\alpha} \nu_\alpha &= \frac{\kappa_1}{N}\sum_{\alpha, \beta} \frac{\psi_{\alpha\beta}}{T_\alpha} \sin(\theta_\beta-\theta_\alpha) = -\frac{\kappa_1}{N}\sum_{\alpha, \beta} \frac{\psi_{\alpha\beta}}{T_\beta} \sin(\theta_\beta-\theta_\alpha)  \\
&=  \frac{\kappa_1}{2N}\sum_{\alpha, \beta} \psi_{\alpha\beta} \Big( \frac{1}{T_\alpha} -  \frac{1}{T_\beta}  \Big)  \sin(\theta_\beta-\theta_\alpha). 
 \end{aligned}
 \end{align*}
This yields
\[   \Big| \frac{d}{dt} \Big(  \sum_{\alpha} \theta_\alpha - t \sum_{\alpha} \nu_\alpha \Big)  \Big| \leq \frac{\kappa_1 \max \psi_{\alpha \beta} N}{2 (T_1^{in})^2} {\mathcal D}(T(t)) \lesssim e^{-\Lambda t} \quad \mbox{as $t \to \infty$}.    \]
\end{proof}
\section{Emergence of phase synchronization} \label{sec:3}
\setcounter{equation}{0} 
In this section, we study the emergent dynamics of the TK model. Note that the temperature dependence in $\eqref{B-8}_1$ in the phase evolution breaks down the conservation law structure for the Kuramoto model, and thus several mathematical difficulties arise in the analysis of emergent dynamics. 
For instance, due to the lack of conservation law, we cannot explicitly determine the asymptotic limit of $\{\theta_{\alpha} \}$ a priori. Moreover, it is clear that the given system \eqref{B-8} does not have a gradient flow structure as Kuramoto model \cite{v-W}, which is one of key structure to analyze the emergent dynamics of the Kuramoto model for a generic initial data as in \cite{H-K-R}. Although the dynamics of TK model is somehow different with Kuramoto model, we basically follow the previous results for Kuramoto model and try to prove their analogous statements, since Kuramoto model approximates the TK model asymptotically. In that sense, we first recall definitions on the asymptotic states of Kuramoto model.
\begin{definition}
	Let $\Theta=(\theta_1,\cdots,\theta_N)$ be a solution to \eqref{KM}.
	\begin{enumerate}
		\item We say that the solution $\Theta$ exhibits a(n asymptotic) complete phase synchronization if \[\lim_{t \to \infty} |\theta_{\alpha}(t)-\theta_\beta(t)|=0,\quad \forall~ \alpha,\beta=1,\cdots,N. \]
		\item We say that the solution $\Theta$ exhibits a(n asymptotic) phase locking if 
		\[\lim_{t \to \infty} (\theta_\alpha(t)-\theta_\beta(t))~\text{exists},\quad \forall~\alpha,\beta=1,\cdots,N. \]
	\end{enumerate}
\end{definition}
On the other hand, since \eqref{KM} has a rotational symmetry and the conservation law $\sum_{\alpha}\dot\theta_\alpha=\sum_\alpha\nu_\alpha$, one might consider a rotating frame $\theta\longrightarrow\theta+\bar\nu t$, where $\bar\nu:=\frac{1}{N}\sum_{\alpha}\nu_\alpha$. In this rotating frame, the total sum of phases $\sum_\alpha\theta_\alpha$ is conserved, and the above two notions reduce to the existence of limit phases $\theta_{\alpha}^\infty=\lim_{t \to \infty}\theta_\alpha(t)$ which is identical in $\alpha$ for complete synchronization and nonidentical for phase locking, respectively. We will also use these definitions to describe the asymptotic states of TK model.
\subsection{Homogeneous TK ensemble} \label{sec:3.1}   In this subsection, we consider a homogeneous ensemble with the same natural frequency:
\[ \nu_\alpha = \nu, \quad \alpha = 1, \cdots, N. \]
Without loss of generality, we may assume $\nu= 0$, or we may consider a rotating frame with $\theta \longrightarrow \theta+ \nu t$. Then, $\{(\theta_\alpha, T_\alpha)\}$ satisfies
\begin{equation} \label{C-1}
\begin{cases}
\displaystyle \dot{\theta}_\alpha = \frac{\kappa_1}{N}\sum_{\beta} \frac{\psi_{\alpha\beta}}{T_\alpha} \sin(\theta_\beta-\theta_\alpha), \quad t > 0,~~\alpha = 1, \cdots, N, 
 \\
\displaystyle \dot{T}_\alpha = \frac{\kappa_2}{N}\sum_{\beta} \frac{\zeta_{\alpha\beta}T_*^2}{(T_*^2+\eta^2 T_\alpha)}\left(\frac{1}{T_\alpha}-\frac{1}{T_\beta}\right).
\end{cases}
\end{equation}
According to the notation introduced in Section \ref{sec:1}, for a given phase vector $\Theta = \{(\theta_\alpha)\}$, we denote a phase diameter ${\mathcal D}(\Theta)$ as:
\[ {\mathcal D}(\Theta) := \max_{\alpha, \beta} |\theta_\alpha - \theta_\beta|. \]
Then, in the following lemma, we show that the phase diameter ${\mathcal D}(\Theta)$ is contractive along the TK flow \eqref{C-1}. 
\begin{lemma}\label{L3.1}
    Let $\{(\theta_{\alpha},T_\alpha) \}_{\alpha=1}^N$ be a solution to the system \eqref{C-1}, where the  initial data $\{ (\theta_{\alpha}^{in},T_\alpha^{in}) \}_{\alpha=1}^N$ and system parameter $\kappa_1,\kappa_2,\psi_{\alpha \beta},\zeta_{\alpha \beta}$ satisfy
    \[\kappa_1,\kappa_2>0,\quad \psi_{\min}>0,\quad \zeta_{\min}>0,\quad 0<T_1^{in}\leq \cdots\leq T_N^{in}<\infty \quad \mbox{and} \quad {\mathcal D}(\Theta^{in}) <\pi. \]
    Then, ${\mathcal D}(\Theta)$ is monotonically decreasing in $t$:
    \[ {\mathcal D}(\Theta(t)) \leq {\mathcal D}(\Theta^{in}), \quad \forall~t \geq 0. \]  
\end{lemma}
\begin{proof} First, let $\varepsilon$ be any positive number satisfying 
    \[  {\mathcal D}(\Theta^{in})+\varepsilon<\pi, \]
    and suppose that the following set is nonempty:
    \begin{equation}\label{set}
    \mathcal{S}_\varepsilon:=\left\{t >0:~ {\mathcal D}(\Theta(t)) >   {\mathcal D}(\Theta^{in})+\varepsilon \right\}.
    \end{equation}
    Then, the infimum  $t_\varepsilon:=\inf\mathcal{S}_\varepsilon$ is finite and strictly positive, since $\mathcal{S}_\varepsilon$ is nonempty and $\{(\theta_{\alpha}) \}$ is continuous.\\
    
    \noindent Similar to Lemma \ref{L2.1}, we consider a time sequence $\{\tau_m\}_{m=0}^n$ 
    such that 
    \[0=\tau_0<\tau_1<\cdots<\tau_n=t_\varepsilon, \]
    and there exists a unique maximal and minimal index $M', m'$ on each time interval $(\tau_{i-1},\tau_i)$: 
    \[\theta_{M'}(t)=\max_{\alpha} \theta_{\alpha}(t),\quad \theta_{m'}(t)=\min_{\alpha}\theta_{\alpha}(t),\quad t\in (\tau_{i-1},\tau_i). \] 
    Then, we have the following estimates along each interval $(\tau_{i-1},\tau_i)$:
    \begin{equation}\label{df}
    \frac{d}{dt}(\theta_{M'}-\theta_{m'})=\frac{\kappa_1}{N}\sum_{\beta}\frac{\psi_{M' \beta}}{T_{M'}}\sin(\theta_{\beta}-\theta_{M'})-\frac{\kappa_1}{N}\sum_{\beta}\frac{\psi_{m' \beta}}{T_{m'}}\sin(\theta_{\beta}-\theta_{m'}).
    \end{equation}
    On the other hand, as long as $t$ is smaller than $t_\varepsilon$, we know that the maximal difference 
    $\theta_{M'}-\theta_{m'}$ is smaller than $\pi$:
     \[\theta_{M'}(t)-\theta_{m'}(t)\leq  {\mathcal D}(\Theta^{in}) +\varepsilon<\pi,\quad \forall t<t_\varepsilon. \]
     Thus, the right-hand side of \eqref{df} is strictly negative for all but finitely many $t$ along $(0,t_\varepsilon)$, so that 
     \[\theta_{M'}(t_\varepsilon)-\theta_{m'}(t_\varepsilon)\leq  {\mathcal D}(\Theta^{in}). \]
     However, we also have a lower bound of maximal difference at time $t_\varepsilon$ from \eqref{set}:
     \[\theta_{M'}(t_\varepsilon)-\theta_{m'}(t_\varepsilon)\geq   {\mathcal D}(\Theta^{in}) +\varepsilon, \]
     which gives a contradiction. Therefore, we deduce the emptiness of the set $\mathcal{S}_\varepsilon$ for every small $\varepsilon$ and conclude our desired result.
\end{proof}
Next, we show the emergence of phase convergence for the homogeneous ensemble in two cases. 
\begin{itemize}
\item
Case A:~(Emergence of complete phase synchronization): We use Lemma \ref{L3.1} and extremal phase argument to show that if $\mathcal{D}(\Theta^{in})<\pi$, there exists a positive constant $\tilde \Lambda$ and a limit phase $\theta^\infty$ such that 
\[ \|\theta_\alpha(t)-\theta^\infty\|_{\infty} \lesssim e^{-{\tilde \Lambda} t} \quad \mbox{as $t \to \infty$},\quad \alpha=1,\cdots,N. \]

\vspace{0.2cm}
\item
Case B:~(Convergence of $\theta_\alpha$): We use the result of Proposition \ref{P2.4} and perturbative result of the gradient flow to show that even without any condition on initial data, there exists asymptotic phase configuration $\Theta^{\infty} := (\theta_1^{\infty}, \cdots, \theta_N^{\infty})$ such that 
\[ \lim_{t \to \infty}  |\theta_\alpha(t) - \theta^{\infty}_\alpha| = 0, \quad \alpha=1,\cdots,N. \]
\end{itemize} 

\begin{theorem} \label{T3.0}
    Let $\{ (\theta_{\alpha},T_\alpha) \}_{\alpha=1}^N$ be a solution to \eqref{C-1}, where the  initial data $\{ (\theta_{\alpha}^{in},T_\alpha^{in}) \}_{\alpha=1}^N$ and system parameter $\kappa_1,\kappa_2,\psi_{\alpha \beta},\zeta_{\alpha \beta}$ satisfy
    \[\kappa_1,\kappa_2>0,\quad \psi_{\min}>0,\quad \zeta_{\min}>0,\quad 0<T_1^{in}\leq \cdots\leq T_N^{in}<\infty \quad \mbox{and} \quad {\mathcal D}(\Theta^{in}) <\pi. \]
    Then, one has an exponential synchronization: there exists a constant $\theta^\infty$ satisfying  
    \begin{equation}\label{rate} 
        |\theta_\alpha(t)-\theta^\infty| \lesssim e^{-{\tilde \Lambda} t}  ~ \mbox{as $t \to \infty$}, \quad \forall~{\tilde \Lambda} <\frac{\kappa_1\psi_{\min}}{T^\infty},~ \alpha=1,\cdots,N.
    \end{equation} 
\end{theorem}
\begin{proof}
   Similar to Lemma \ref{L3.1}, there exists a time sequence $\{t_n \}_{n\geq 0}$ such that 
    \[0=t_0<t_1<t_2<\cdots, \]
    and unique maximal and minimal indices $M'$ and $m'$ for $\theta$ can be defined on each time interval $(\tau_{i-1},\tau_i)$. Then, on each time interval $(\tau_{i-1},\tau_i)$, we obtain:
    \[\begin{aligned}
    \frac{d}{dt}(\theta_{M'}-\theta_{m'})&=\frac{\kappa_1}{N}\sum_{\beta}\frac{\psi_{M' \beta}}{T_{M'}}\sin(\theta_{\beta}-\theta_{M'})-\frac{\kappa_1}{N}\sum_{\beta}\frac{\psi_{m' \beta}}{T_{m'}}\sin(\theta_{\beta}-\theta_{m'})\\
    &\leq \frac{\kappa_1}{N}\sum_{\beta}\frac{\psi_{M' \beta}}{T_N^{in}}\sin(\theta_{\beta}-\theta_{M'})-\frac{\kappa_1}{N}\sum_{\beta}\frac{\psi_{m' \beta}}{T_N^{in}}\sin(\theta_{\beta}-\theta_{m'})\\
    &\leq -\frac{\kappa_1\psi_{\min}}{NT_{N}^{in}}\sum_{\beta}\left[\sin(\theta_{M'}-\theta_{\beta})+\sin(\theta_{\beta}-\theta_{m'})\right]\\
    &\leq -\frac{\kappa_1\psi_{\min}}{T_{N}^{in}}\sin(\theta_{M'}-\theta_{m'})=:-{\tilde \Lambda} \sin(\theta_{M'}-\theta_{m'}),
    \end{aligned}\]
    where we used Lemma \ref{L3.1} and $0\leq \theta_{M'}-\theta_{\beta},~\theta_{\beta}-\theta_{m'}<\pi$ in the last inequality. Therefore, we have 
    \[{(\theta_{M'}-\theta_{m'})(t)}\leq 2\tan \frac{(\theta_{M'}-\theta_{m'})(t)}{2}\leq 2\tan \frac{(\theta_{M'}-\theta_{m'})(0)}{2}e^{-{\tilde \Lambda} t}.  \]
    Moreover, we can also obtain the convergence of each $\theta_\alpha$ to their limit $\theta_\alpha^\infty$ by estimating the exponential decay of derivatives: since $\dot \theta_\alpha$ decays exponentially,
    \[
    |\dot\theta_{\alpha}| =\left|\frac{\kappa_1}{N}\sum_{\beta}\frac{\psi_{\alpha \beta}}{T_{\alpha}}\sin(\theta_{\beta}-\theta_{\alpha})\right| \leq \frac{\kappa_1\psi_{\max}}{T_1^{in}}\sin {\mathcal D}(\Theta(t)) \lesssim e^{-{\tilde \Lambda} t},
    \]    
    we have 
    \[|\theta_{\alpha}^\infty-\theta_\alpha(t)|\leq \int_{t}^{\infty}|\dot\theta_{\alpha}(s)|\lesssim e^{-{\tilde \Lambda} t},\quad \alpha=1,\cdots,N. \]
    Then, as the phase diameter $\mathcal{D}(\Theta)$ converges to zero, these all $\theta_\alpha^\infty$s are indeed equal. Finally, we use the convergence of $T_\alpha$ to $T^\infty$ and change the starting time if necessary to generalize decay rate $\tilde\Lambda$ to any constant smaller than $\frac{\kappa_1\psi_{\min}}{T^\infty}$.
\end{proof}
\begin{remark}
For the Kuramoto model, the complete synchronization of  identical oscillators (i.e., all phases converge to a common constant) and their asymptotic limits can be immediately followed from the vanishing of phase diameter and the conservation of total phase. However, we here provided the existence of an asymptotic limit $\theta_\alpha^\infty$ by estimating the exponential decay rates of the derivatives $\dot\theta_{\alpha}$ directly. Indeed, one can also show the convergence of $\theta_{\alpha}$    using the exponential convergence of total sum $\sum_{\alpha} \theta_\alpha$ (see Corollary \ref{C2.1}) and the asymptotic vanishing property of phase diameter, but then the decay rates of total phase and phase diameter are different. 
\end{remark}

Next, we reformulate the TK model as an exponential perturbation of a gradient flow and then use the perturbation theory of a gradient flow to show that the convergence of phase-temperature state $(\theta_\alpha, T_\alpha)$ for every initial data.  We first recall a preparatory lemma for the perturbation of a gradient flow system without proof.
\begin{lemma}\label{L3.2}
\emph{\cite{H-J-K-P-Z}}
    Let $X=X(t) \in \bbr^d$ be a solution to a gradient flow-like system
    \[\dot X(t)=-\nabla_{X}V(X)+ F(t),  \]
    where $V$ is a potential function and the forcing function $F$ is assumed to decay to zero exponentially fast. Moreover, suppose that there exists a compact set $\mathcal{K} \subset \bbr^d$ such that 
    \begin{enumerate}
        \item The flow $\{X(t) \}_{t\geq 0}$ is contained in some compact set $\mathcal{K}$.
        \item $V$ is analytic on $\mathcal{K}$.
        \item $|\nabla_{X}V|^2$ is uniformly continuous in time.
        \end{enumerate}
    Then, there exists $X^\infty\in\mathcal{K}$ such that $\lim_{t\to\infty}X(t)=X^\infty.$
\end{lemma}
\begin{proof}
We refer to \cite{H-J-K-P-Z} for a proof.
\end{proof}
\begin{remark}
    If the potential $V$ is periodic, we consider $\mathbb{T}^N$ as the ambient manifold, and then the condition (1) in Lemma \ref{L3.2} immediately satisfied.
\end{remark}
Next, we show the existence of constant asymptotic states for the phase and temperature in the following theorem.
\begin{theorem}\label{T3.1}
Let $\{ (\theta_{\alpha},T_\alpha) \}_{\alpha=1}^N$ be a solution to \eqref{C-1} with initial data $\{ (\theta_{\alpha}^{in},T_\alpha^{in}) \}_{\alpha=1}^N$, and
suppose that the coupling strengths, communication weights  and the initial data satisfy 
\[ \kappa_1,~\kappa_2 > 0,\quad \psi_{\min}>0,\quad \zeta_{\min}>0, \quad 0<T_1^{in}\leq \cdots\leq T_N^{in}<\infty.  \] Then, there exists a constant phase-temperature state $(\Theta^\infty,T^\infty)=(\theta_1^\infty,\cdots,\theta_N^\infty,T^\infty)$ such that 
    \[\lim_{t\to\infty}\theta_\alpha(t)=\theta_\alpha^\infty,\quad  \lim_{t\to\infty}T_\alpha(t)=T^\infty,\quad \forall \alpha=1,\cdots,N, \]
    where the asymptotic states $\{ (\theta_\alpha^{\infty}, T_\alpha^{\infty}) \}$ satisfy 
    \begin{equation} \label{rel}
    \sum_{\beta} \psi_{\alpha \beta}\sin(\theta_\beta^\infty-\theta_\alpha^\infty)=0,\quad  T^\infty+\frac{\eta^2}{2T_*^2}{T^\infty}^2=\frac{1}{N}\sum_{\beta} \left({T^{in}_\beta}+\frac{\eta^2}{2T_*^2}{T^{in}_\beta}^2\right),~ \forall~\alpha \in \{1, \cdots, N \}. 
\end{equation}
\end{theorem}
\begin{proof}
    We first rewrite $\eqref{C-1}_1$ to gradient-like flow structure in Lemma \ref{L3.2}. For a phase configuration $\Theta=(\theta_1,\cdots,\theta_N)\in\mathbb{T}^N$, let $V:\mathbb{T}^N\to\mathbb{R}$ be an analytic potential function:  
    \[V(\Theta) :=\frac{\kappa_1}{2NT^\infty}\sum_{\alpha, \beta}\psi_{\alpha \beta} \left(1-\cos(\theta_{\alpha}-\theta_\beta )\right), \] 
    where the constant $T^\infty$ is given by Proposition \ref{P2.4}. \newline
    
    \noindent Note that equation \eqref{C-1} can be written as 
    \begin{equation}\label{fluct}
    \dot\Theta=-\nabla_{\Theta}V(\Theta)+F(t), \quad t > 0, 
    \end{equation}
    where the forcing function $F = (f_1, \cdots, f_N)$ is explicitly given as follows:
\begin{equation} \label{C-1-1}
 f_\alpha(t)=\frac{\kappa_1}{N}\sum_{\beta} \frac{T^\infty-T_\alpha(t)}{T_\alpha(t) T^\infty}\psi_{\alpha \beta}\sin(\theta_{\beta}(t)-\theta_{\alpha}(t)). \end{equation}
 On the other hand, the forcing function $f_\alpha$ decays to zero due to Proposition \ref{P2.4}: 
    \[|f_\alpha(t)|\leq \frac{\kappa_1}{N}\sum_{\beta}\frac{|T^\infty-T_\alpha(t)|}{T_\alpha(t)T^\infty}\psi_{\alpha\beta}=\mathcal{O}(e^{-\Lambda t}),\quad \Lambda < \frac{\kappa_2\zeta_{\min}T_*^2}{{{T^\infty}^2}(T_*^2+\eta^2{T^\infty})}.  \]
    Therefore, it suffices to verify the condition $(3)$ in Lemma \ref{L3.2}. To prove the uniform continuity of $|\nabla_{\Theta}V|^2$ in $t$, we find a finite uniform bound of its $t$-derivative.
    More precisely, the $t$-derivative can be estimated by substituting $\dot\theta_\alpha$ in \eqref{C-1} into $\frac{d}{dt}|\partial_{\theta_\alpha}V|^2$:
    \[\begin{aligned}
    \left|\frac{d}{dt}|\partial_{\theta_\alpha}V|^2\right|&=\left|\frac{\kappa_1^2}{N^2{T^\infty}^2}\frac{d}{dt}\left(\sum_{\beta}\psi_{\alpha\beta}\sin(\theta_{\beta}-\theta_{\alpha} )\right)^2\right|\\
    &=\frac{2\kappa_1^2}{N^2{T^\infty}^2}\left|\sum_{\beta}\psi_{\alpha\beta}\sin(\theta_{\beta}-\theta_{\alpha} )\right|\left|\sum_{\beta}\psi_{\alpha\beta}\cos(\theta_{\beta}-\theta_{\alpha} )(\dot\theta_{\beta}-\dot\theta_{\alpha})\right|\\
    &\leq \frac{2\kappa_1^2\psi_{\max}^2}{N{T^\infty}^2}\sum_{\beta}\left|\dot\theta_{\beta}-\dot\theta_{\alpha} \right| \leq \frac{2\kappa_1^2\psi_{\max}^2}{N{T^\infty}^2}\cdot 2N\cdot\frac{\kappa_1\psi_{\max}}{T_1^{in}}<\infty.
    \end{aligned} \]
    Now, the convergence of the solution $(\theta_{\alpha},T_\alpha)$ to a constant $(\theta_\alpha^\infty,T^\infty)$ is provided as a consequence of Lemma \ref{L3.2}, and as $f(t)$ and $\nabla_{\Theta}V(\Theta(t))$ converges, we also obtain the convergence of $t$-derivative $\dot{\Theta}(t)$. Since the only possible limit for $\dot\Theta$ compatible with the convergence of $\Theta$ is ${\bf 0}$, we have 
    \[\nabla_{\Theta}V(\Theta^\infty)=\lim_{t\to\infty}\dot\Theta(t)=0. \]
This implies the relation \eqref{rel} for the limit $(\Theta^\infty,T^\infty)$.
 \end{proof}

Although we can derive the existence of limit state $(\theta_\alpha^\infty,T^\infty)$ for homogeneous ensemble, it is still not clear from \eqref{rel} whether the phases $\theta_\alpha$ indeed exhibits a complete phase synchronization $\theta_1^\infty=\cdots=\theta_N^\infty$ or not.  In \cite{H-K-R}, authors provided that for the original Kuramoto model $(\psi_{\alpha\beta}\equiv 1,~ T_\alpha\equiv T_*)$, the limit phases $\Theta^\infty=(\theta_1^\infty,\cdots,\theta_N^\infty)$ can be written as 
\[\Theta^\infty=~ \text{either}\quad (\phi^\infty,\cdots,\phi^\infty)\quad \text{or}\quad (\phi^\infty,\cdots,\phi^\infty, \phi^\infty+\pi), \]
whenever the initial configuration $\Theta^{in}$ satisfies 
\[\theta_\alpha^{in}\neq \theta_\beta^{in}~\text{for}~\alpha\neq \beta,\quad R^{in}=\left|\frac{1}{N}\sum_{\alpha} e^{i\theta_\alpha^{in}} \right|>0. \]

We will now introduce an analogous result for the homogeneous TK model \eqref{C-1} with the following preparatory lemma.

\begin{lemma}\label{L3.4}
    Let $\{ (\theta_{\alpha},T_\alpha) \}_{\alpha=1}^N$ be a solution to \eqref{C-1} with the initial data $\{(\theta_{\alpha}^{in},T_\alpha^{in})\}_{\alpha=1}^N$ satisfying  
    \[\kappa_1, \kappa_2>0,\quad \psi_{\min}\geq0,\quad 0<T_1^{in}\leq \cdots\leq T_N^{in}<\infty. \]
    Then, we have
    \[
    \frac{d}{dt}\left(\sum_{\alpha, \beta}\psi_{\alpha\beta}\cos(\theta_\alpha-\theta_\beta)\right)\geq0.
    \]
\end{lemma}
\begin{proof} By direct calculation, the above target functional can be rewritten as a sum of squares as below:
    \begin{align*}
    \begin{aligned}
    &\frac{d}{dt}\left(\sum_{\alpha, \beta}\psi_{\alpha\beta}\cos(\theta_\alpha-\theta_\beta)\right)=-\sum_{\alpha, \beta}\psi_{\alpha\beta}\sin(\theta_\alpha-\theta_\beta)(\dot{\theta}_\alpha-\dot{\theta}_\beta)\\
    & \hspace{0.5cm} =-2\sum_{\alpha,\beta}\psi_{\alpha\beta}\sin(\theta_\alpha-\theta_\beta)\dot{\theta}_\alpha=-\frac{2\kappa_1}{N}\sum_{\alpha,\beta,\gamma}\frac{1}{T_\alpha}\psi_{\alpha\beta}\psi_{\alpha\gamma}\sin(\theta_\alpha-\theta_\beta)\sin(\theta_\gamma-\theta_\alpha)\\
    &\hspace{0.5cm} =\frac{2\kappa_1}{N}\sum_{\alpha}\frac{1}{T_\alpha}\left(\sum_\beta \psi_{\alpha\beta}\sin(\theta_\alpha-\theta_\beta)\right)^2\geq0.
    \end{aligned}
    \end{align*}
\end{proof}

\begin{corollary}\label{bipolar}
    Let $\{ (\theta_{\alpha},T_\alpha) \}_{\alpha=1}^N$ be a solution to \eqref{C-1} with the initial data $\{(\theta_{\alpha}^{in},T_\alpha^{in})\}_{\alpha=1}^N$ satisfying  
    \[\kappa_1, \kappa_2>0,\quad  0<T_1^{in}\leq\cdots\leq T_N^{in}<\infty,\quad  R^{in}:=\left|\frac{1}{N}\sum_{\alpha} e^{{\mathrm i}\theta_\alpha^{in}} \right|>0, \quad \psi_{\alpha \beta}\equiv1\quad  \forall~ \alpha, \beta. \]
    Then, there exists a constant $\phi^\infty\in\mathbb{R}$ such that asymptotic phases $(\theta_\alpha^\infty)_\alpha$ satisfies  
    \[ \mbox{either}~~\theta_\alpha=\phi^\infty\quad \text{or}\quad \phi^\infty+\pi,\quad \alpha=1,\cdots,N. \]
\end{corollary}
\begin{proof}
    The initial condition $R^{in}>0$ implies that 
    \[\begin{aligned}
    |R^{in}|^2=\left(\frac{1}{N}\sum_{\alpha} e^{{\mathrm i} \theta_\alpha^{in}}\right)\overline{\left(\frac{1}{N}\sum_{\alpha} e^{{\mathrm i} \theta_\alpha^{in}}\right)}
    =\frac{1}{N^2}\sum_{\alpha, \beta}e^{{\mathrm i}(\theta_\alpha^{in}-\theta_\beta^{in})}=\frac{1}{N^2}\sum_{\alpha, \beta}\cos(\theta_\alpha^{in}-\theta_\beta^{in})
    \end{aligned} \]
    is strictly positive. Therefore, as $R(\Theta)^2:=\frac{1}{N^2}\sum_{\alpha, \beta} \cos(\theta_\alpha-\theta_\beta)$ is nondecreasing from Lemma \ref{L3.4}, we can deduce that $R(\Theta (t))$ converges to a positive constant $R^\infty:=\left|\frac{1}{N}\sum_{\alpha} e^{{\mathrm i} \theta_\alpha^\infty}\right|$ smaller than 1, and in particular 
    \[\frac{1}{N}\sum_{\alpha} e^{{\mathrm i} \theta_\alpha^\infty}\neq 0. \]
    
    \noindent On the other hand, for $\psi_{\alpha \beta}\equiv 1$, the first relation in \eqref{rel} becomes 
    \[\sum_{\beta = 1}^N\sin(\theta_\beta^\infty-\theta_\alpha^\infty)=0,\quad \alpha=1,\cdots,N. \]
    Therefore, for every $\alpha,\beta=1,\cdots,N$, we have 
    \[\sum_{\gamma = 1}^Ne^{{\mathrm i}(\theta_\gamma^\infty-\theta_\alpha^\infty)}\in\mathbb{R}-\{0\},\quad  e^{{\mathrm i}(\theta_\beta^\infty-\theta_\alpha^\infty)}=\left(\sum_{\gamma = 1}^Ne^{{\mathrm i}(\theta_\gamma^\infty-\theta_\alpha^\infty)}\right){\left(\sum_{\gamma = 1}^Ne^{{\mathrm i}(\theta_\gamma^\infty-\theta_\beta^\infty)}\right)}^{-1}\in\mathbb{R}, \]
     which concludes the desired result.
    
\end{proof}

\subsection{Heterogeneous TK ensemble} In this subsection, we consider a heterogeneous ensemble of TK oscillators:
\begin{equation} \label{C-2}
\begin{cases}
\displaystyle \dot{\theta}_\alpha = \nu_\alpha+ \frac{\kappa_1}{N}\sum_{\beta} \frac{\psi_{\alpha\beta}}{T_\alpha} \sin(\theta_\beta-\theta_\alpha),~~ \alpha=1,\cdots,N,~ t>0,\\
\displaystyle \dot{T}_\alpha = \frac{\kappa_2}{N}\sum_{\beta} \frac{\zeta_{\alpha\beta}T_*^2}{(T_*^2+\eta^2 T_\alpha)}\left(\frac{1}{T_\alpha}-\frac{1}{T_\beta}\right).
\end{cases}
\end{equation}

\vspace{0.2cm}

For the emergent dynamics of \eqref{C-2}, we take the following three steps: 
\begin{itemize}
\item
Step A:~Under a priori uniform boundedness of phases, we show that the TK model can be rewritten as an exponential perturbation of a gradient flow, and then using the perturbation theory of the gradient flow in Lemma \ref{L3.2}, we show that the phase field approaches the constant phase field asymptotically (see Lemma \ref{L3.3}).

\vspace{0.2cm}

\item
Step B:~Under suitable smallness assumption on the initial phase diameter, we show that the phases are uniformly bounded in time, and asymptotically,  the phase diameter is bounded by the quantity inversely proportional to $\kappa_1^{-1}$, i.e., the formation of practical synchronization (see Proposition \ref{P3.2}).

\vspace{0.2cm}

\item
Step C:~By collecting all the estimates in Step A and Step B, we show that the phase configuration evolves toward the phase-locked state (see Theorem \ref{T3.2}).
\end{itemize}

\medskip

In the sequel, we perform the above three steps. 

\begin{lemma}\label{L3.3}
        Let $\{(\theta_{\alpha},T_\alpha) \}_{\alpha=1}^N$ be a solution to \eqref{C-2} with the initial data $\{(\theta_{\alpha}^{in},T_\alpha^{in})\}_{\alpha=1}^N$ satisfying positivity and a priori condition:
    \[\kappa_1,\kappa_2>0,\quad \psi_{\min}>0,\quad \zeta_{\min}>0,\quad \min_{\alpha}T_\alpha^{in}>0, \quad  \sup_{0 \leq t < \infty}\max_{\alpha}|\theta_{\alpha}(t)| \leq D < \infty. \]
    Then, there exists a unique asymptotic state $(\theta_\alpha^\infty,T^\infty)$:
    \[\lim_{t\to\infty}\theta_{\alpha}(t)=\theta_\alpha^\infty \quad \mbox{and} \quad \lim_{t\to\infty} T_\alpha(t)=T^\infty, \]
    where the asymptotic state $(\theta_\alpha^\infty,T^\infty)$ satisfies the following relations for all $\alpha$:
    \begin{equation}\label{rel2}
    T^\infty+\frac{\eta^2}{2T_*^2}{T^\infty}^2=\frac{1}{N}\sum_{\alpha} \left({T^{in}_\alpha}+\frac{\eta^2}{2T_*^2}{T^{in}_\alpha}^2\right), \quad  \nu_\alpha +\frac{\kappa_1}{NT^\infty}\sum_{\beta}\psi_{\alpha \beta}\sin(\theta_{\beta}^\infty-\theta_{\alpha}^\infty)=0.
    \end{equation}
\end{lemma}
\begin{proof}
    We again apply the result of Lemma \ref{L3.2} to prove the convergence result. In this case, however, we cannot employ the periodicity of the potential $V$ as in Theorem \ref{T3.1}. Therefore, we here regard phase configuration $\Theta$ as a vector in $\mathbb{R}^N$ and use the \textit{a priori} boundedness of $\Theta$ to satisfy the condition $(1)$ of Lemma \ref{L3.2}. For a phase configuration $\Theta=(\theta_1,\cdots,\theta_N)\in\mathbb{R}^N$, let $\widetilde V:\mathbb{R}^N\to\mathbb{R}$ be an analytic potential function  
    \[\widetilde V(\Theta)=-\sum_{\alpha}\nu_\alpha\theta_{\alpha}+\frac{\kappa_1}{2NT^\infty}\sum_{\alpha, \beta}\psi_{\alpha \beta} \left(1-\cos(\theta_{\alpha}-\theta_\beta )\right), \] 
    where the constant $T^\infty$ is given by Proposition \ref{P2.4}.\\
    
    \noindent Then, the equation \eqref{C-2} can also be written as a perturbation of a gradient flow:
    \begin{equation}\label{gralike3}
    \dot\Theta=-\nabla_{\Theta}\widetilde{V}(\Theta)+F(t),
    \end{equation}
    where $F = F(t)$ is the same fluctuation function used in \eqref{fluct}. Thus, we again have a gradient flow-like system \eqref{gralike3} with exponentially converging fluctuation $F$ and analytic potential $\widetilde{V}$. Now, since the solution trajectory $\{\Theta(t)\}_{t\geq 0}$ is contained in a compact set $[-D,D]^N$, the uniform boundedness of $\frac{d}{dt}|\partial_{\theta_\alpha}\widetilde{V}|^2$ will conclude our desired convergence result. We obtain the uniform boundedness of $\frac{d}{dt}|\partial_{\theta_\alpha}\widetilde{V}|^2$ by direct calculations:
    \[\begin{aligned}
    \left|\frac{d}{dt}|\partial_{\theta_\alpha}\widetilde V|^2\right|&=\left|\frac{d}{dt}\left(\nu_\alpha+\frac{\kappa_1}{N{T^\infty}}\sum_{\beta}\psi_{\alpha\beta}\sin(\theta_{\beta}-\theta_{\alpha} )\right)^2\right|\\
    &=\frac{2\kappa_1}{N{T^\infty}}\left|\nu_\alpha+\frac{\kappa_1}{NT^\infty}\sum_{\beta}\psi_{\alpha\beta}\sin(\theta_{\beta}-\theta_{\alpha} )\right|\left|\sum_{\beta}\psi_{\alpha\beta}\cos(\theta_{\beta}-\theta_{\alpha} )(\dot\theta_{\beta}-\dot\theta_{\alpha})\right|\\
    &\leq \frac{2\kappa_1\psi_{\max}}{N{T^\infty}}\left(|\nu_\alpha|+\frac{\kappa_1\psi_{\max}}{T^\infty} \right)\sum_{\beta}\left|\dot\theta_{\beta}-\dot\theta_{\alpha} \right|\\
    &\leq \frac{2\kappa_1\psi_{\max}}{N{T^\infty}}\left(|\nu_\alpha|+\frac{\kappa_1\psi_{\max}}{T^\infty} \right)\sum_{\beta}\left(|\nu_\alpha|+|\nu_\beta|+\frac{2\kappa_1\psi_{\max}}{T_1^{in}} \right)<\infty.
    \end{aligned} \]
    Finally, we use the analogous argument as in Theorem \ref{T3.1} to deduce the relation \eqref{rel2}:
    \[\nabla_{\Theta}\widetilde{V}(\Theta^\infty)=0. \]
\end{proof}

\begin{remark}
    If the state $(\Theta^\infty,T^\infty)$ satisfy the relation \eqref{rel2}, we have 
    \begin{equation}\label{necess}
    \sum_{\alpha} \nu_{\alpha}=-\frac{\kappa_1}{NT^\infty}\sum_{\alpha,\beta}\psi_{\alpha \beta}\sin(\theta_{\beta}^\infty-
    \theta_{\alpha}^\infty)=0. 
    \end{equation}
    Therefore, \eqref{necess} is a necessary condition for the \textit{a priori} uniform boundedness of $\Theta$.
\end{remark}

Now we have the boundedness of the solution for the heterogeneous TK model.
\begin{proposition}\label{P3.2}
Let $\{ (\theta_\alpha, T_\alpha) \}_{\alpha=1}^N$ be a solution to \eqref{C-2} with initial data $\{ (\theta_\alpha^{in}, T_\alpha^{in}) \}_{\alpha=1}^N$ satisfying
\begin{equation}\label{cond1}
\kappa_1,\kappa_2>0,\quad \psi_{\min}>0,\quad \zeta_{\min}>0,\quad 0<T_1^{in}\leq \cdots\leq T_N^{in}<\infty,\quad \sum_\alpha \nu_\alpha=0,
\end{equation}
and
\begin{equation}\label{cond2}
\frac{{\mathcal D}(\nu) T_N^{in}}{\kappa_1\psi_{\min}}<1,\quad {\mathcal D}(\Theta^{in}) <\pi-\theta_*,\quad \theta_*:=\arcsin\left(\frac{{\mathcal D}(\nu) T_N^{in}}{\kappa_1\psi_{\min}}\right)<\frac{\pi}{2}.
\end{equation}
Then, the following assertions hold.
\medskip
\begin{enumerate}
\item
The phase diameter is bounded:
\[
\sup_{0 \leq t < \infty} \mathcal{D}(\Theta(t)) \leq \pi-\theta_*.
\]
\item The phases are asymptotically concentrated to a quarter circle:
\[\limsup_{t\rightarrow\infty} {\mathcal D}(\Theta(t)) \leq \arcsin\left(\frac{{\mathcal D}(\nu) T^\infty}{\kappa_1\psi_{\min}}\right)\leq \theta_*.\]
\item
Each phase is uniformly bounded in time:
\[
\sup_{0 \leq t < \infty}\max_{\alpha}|\theta_{\alpha}(t)|<\infty.
\]
\end{enumerate}
\end{proposition}
\medskip
\begin{proof}
\noindent (i)~ We follow the proof of Lemma \ref{L3.1}. Suppose on the contrary that the following set is nonempty:
\begin{equation}\label{set2}
\mathcal{S}:=\left\{t>0:~ {\mathcal D}(\Theta(t)) \geq  \pi-\theta_* \right\}.
\end{equation}
Then, the infimum $t_*:=\inf\mathcal{S}$ is finite and strictly positive.\\

\noindent Now, we again consider a time sequence $\{\tau_m \}_{m=0}^n$ satisfying
\[0=\tau_0<\tau_1<\cdots<\tau_n=t_*, \]
so that each time interval $(\tau_{i-1},\tau_i)$ attains a unique maximal and minimal $\theta_\alpha$. For each time $t\in(\tau_{i-1},\tau_i)$, we denote these unique extremal indices as $M=M_t$ and $m=m_t$, i.e.,
\[ \theta_M=\max_\alpha \theta_\alpha  \quad \mbox{and} \quad  \theta_m=\min_\alpha \theta_\alpha. \]
For each $i=1,\cdots,n$ and $t\in(\tau_{i-1},\tau_i)$, we then obtain
\begin{align*}
\frac{d}{dt}(\theta_M-\theta_m)&=(\nu_M-\nu_m)+\frac{\kappa_1}{N}\sum_{\beta}\left(\frac{\psi_{M\beta}}{T_M}\sin(\theta_\beta-\theta_M)-\frac{\psi_{m\beta}}{T_m}\sin(\theta_\beta-\theta_m)\right).
\end{align*}
\noindent This yields
\begin{align}\label{adler}
\begin{aligned}
\frac{d}{dt}(\theta_M-\theta_m)&=(\nu_M-\nu_m)+\frac{\kappa_1}{N}\sum_{\beta}\left(\frac{\psi_{M\beta}}{T_M}\sin(\theta_\beta-\theta_M)-\frac{\psi_{m\beta}}{T_m}\sin(\theta_\beta-\theta_m)\right)\\
&\leq {\mathcal D}(\nu)-\frac{\kappa_1}{N}\sum_{\beta}\left(\frac{\psi_{M\beta}}{T_M}\sin(\theta_M-\theta_\beta)+\frac{\psi_{m\beta}}{T_m}\sin(\theta_\beta-\theta_m)\right)\\
&\leq {\mathcal D}(\nu)-\frac{\kappa_1}{N}\sum_{\beta}\left(\frac{\psi_{\min}}{T_N^{in}}\sin(\theta_M-\theta_\beta)+\frac{\psi_{\min}}{T_N^{in}}\sin(\theta_\beta-\theta_m)\right)\\
&< {\mathcal D}(\nu)-\frac{\kappa_1\psi_{\min}}{T^{in}_N}\sin(\theta_M-\theta_m),\quad \forall~t\in(\tau_{i-1},\tau_i),~i=1,\cdots,N.
\end{aligned}
\end{align}
Here, we used 
\[\sin x+\sin y> \sin(x+y)\quad \text{for}\quad 0<x,y<x+y<\pi \]
in the last inequality.\\

\noindent On the other hand, recall that the continuity of $\mathcal{D}(\Theta)$ assures the existence of a small positive number $\delta(<t_*)$ such that 
\[\mathcal{D}(\Theta(t_*))=\pi-\theta_*,\,\, \mathcal{D}(\Theta(t))>\frac{\pi}{2}\quad \text{for}\quad t\in(t_*-\delta,t_*). \]

\noindent Then, it follows from the inequality \eqref{adler} that $\mathcal{D}(\Theta(t))$ is strictly decreasing in for $t\in(t_*-\delta,t_*)$, which contradicts to the minimality of $t_*$. 

\vspace{0.8cm}

\noindent (ii)~Again, we negate the conclusion (2) and derive a contradiction. Suppose that
\begin{equation}\label{negate}
\limsup_{t \to \infty}\mathcal{D}(\Theta(t))>\theta_*. 
\end{equation}  
Then, for sufficiently small $\varepsilon>0$, there exists a time $T_\varepsilon>0$ such that 
\[\mathcal{D}(\Theta(t))\geq\theta_*+\varepsilon,\quad \forall~ t\geq T_\varepsilon.  \]
Since $\mathcal{D}(\Theta(t))$ is always smaller than $\pi-\theta_*$, the derivative of the diameter $\mathcal{D}(\Theta)$ has a negative upper bound $-C$:
\[\begin{aligned}
\frac{d}{dt}\mathcal{D}(\Theta)&< D(\nu)-\frac{\kappa_1\psi_{\min}}{T^{in}_M}\sin\mathcal{D}(\Theta)\\
&\leq D(\nu)-\frac{\kappa_1\psi_{\min}}{T^{in}_M}\min\Big\{\sin(\theta_*+\varepsilon),\sin\mathcal{D}(\Theta(T_\varepsilon)) \Big\}=:-C<0,\quad \forall~ t\geq T_\varepsilon.
\end{aligned}  \]

\noindent However, this contradicts to the assumption \eqref{negate}, as we have 
\[\limsup_{t \to \infty}\mathcal{D}(\Theta(t))\leq\limsup_{t \to \infty} \Big(\mathcal{D}(\Theta(T_\varepsilon))-C(t-T_\varepsilon)\Big)=-\infty. \]

\noindent Therefore, we change the starting time to make $\max_{\alpha}T_\alpha$ closer to $T^\infty$ and apply \[\limsup_{t \to \infty} \mathcal{D}(\Theta(t))\leq \theta_*\] to obtain the desired result. \newline

\noindent (iii)~From Corollary \ref{C2.1}, we know that the average phase $\theta_c:=\frac{1}{N}\sum_{\alpha}\theta_\alpha$ converges exponentially, and therefore uniformly bounded in time. Hence, we have 

\[\begin{aligned}
\sup_{0 \leq t < \infty}|\theta_\alpha(t)|&\leq \sup_{0 \leq t < \infty}|\theta_\alpha(t)-\theta_c(t)|+\sup_{0 \leq t < \infty}|\theta_c(t)|\\
&\leq \sup_{0 \leq t < \infty}\mathcal{D}(\Theta(t))+\sup_{0 \leq t < \infty}|\theta_c(t)| <\infty,\quad \alpha=1,\cdots,N,
\end{aligned} \]
and conclude the desired uniform boundedness.
\end{proof}
Finally, we combine Lemma \ref{L3.3} and Proposition \ref{P3.2} to deduce the emergence of asymptotic equilibrium for the TK model with the distributed natural frequency.
\begin{theorem}\label{T3.2}
Let $\{ (\theta_\alpha, T_\alpha) \}_{\alpha=1}^N$ be a solution to \eqref{C-2} with initial data $\{ (\theta_\alpha^{in}, T_\alpha^{in}) \}_{\alpha=1}^N$ satisfying \eqref{cond1}--\eqref{cond2}.
Then, there exists a constant asymptotic state $(\Theta^\infty,T^\infty)$:
\[\lim_{t\to\infty}\theta_{\alpha}(t)=\theta_\alpha^\infty \quad \mbox{and} \quad \lim_{t\to\infty} T_\alpha(t)=T^\infty, \]
where the limit $(\Theta^\infty,T^\infty)$ satisfies the following relations for every $\alpha$:
\begin{equation}\label{equili}
T^\infty+\frac{\eta^2}{2T_*^2}{T^\infty}^2=\frac{1}{N}\sum_{\alpha} \left({T^{in}_\alpha}+\frac{\eta^2}{2T_*^2}{T^{in}_\alpha}^2\right), \quad  \nu_\alpha +\frac{\kappa_1}{NT^\infty}\sum_{\beta}\psi_{\alpha \beta}\sin(\theta_{\beta}^\infty-\theta_{\alpha}^\infty)=0.
\end{equation}
\end{theorem}

\vspace{0.5cm}

\section{Estimation on the phase limit $\Theta^\infty$} \label{sec:4}
\setcounter{equation}{0}
In this section, we study asymptotic dynamics for the Cauchy problem for the TK model:
\begin{equation} \label{D-1}
\begin{cases}
\displaystyle \dot{\theta}_\alpha = \nu_\alpha+ \frac{\kappa_1}{N}\sum_{\beta} \frac{\psi_{\alpha\beta}}{T_\alpha} \sin(\theta_\beta-\theta_\alpha), \quad t > 0, ~~\alpha=1,\cdots,N,\\
\displaystyle \dot{T}_\alpha = \frac{\kappa_2}{N}\sum_{\beta} \frac{\zeta_{\alpha\beta}T_*^2}{(T_*^2+\eta^2 T_\alpha)}\left(\frac{1}{T_\alpha}-\frac{1}{T_\beta}\right), \\
\displaystyle (\theta_\alpha(0), T_\alpha(0)) = (\theta_\alpha^{in}, T_\alpha^{in}),\quad \sum_{\alpha} \nu_\alpha = 0.
\end{cases}
\end{equation}
Under the positive temperature framework, we have seen that the temperatures converge to the same value (see Proposition \ref{P2.4}): there exists an asymptotic temperature state $T^{\infty}(>0)$ such that 
\[ T^\infty+\frac{\eta^2}{2T_*^2}{T^\infty}^2=\frac{1}{N}\sum_{\alpha} \left({T^{in}_\alpha}+\frac{\eta^2}{2T_*^2}{T^{in}_\alpha}^2\right),\quad \lim_{t \to \infty} | T_\alpha(t) - T^{\infty} | = 0, \quad \alpha = 1, \cdots, N. \]
Thus, dynamics \eqref{D-1} formally reduces to the following Cauchy problem asymptotically:
\begin{equation} \label{D-2}
\begin{cases}
\displaystyle \dot{\phi}_\alpha = \nu_\alpha+\displaystyle\frac{\kappa_1}{N}\sum_{\beta} \frac{\psi_{\alpha\beta}}{T^\infty} \sin(\phi_\beta-\phi_\alpha),~~t > 0,~~\alpha = 1, \cdots, N, \\
\displaystyle  \phi_\alpha(0) = \theta_\alpha^{in}, \quad  \sum_{\alpha} \nu_\alpha = 0.
\end{cases}
\end{equation}
However, although Theorem \ref{T3.2} guarantees the convergence of the phase configuration $\Theta=(\theta_1,\cdots,\theta_N)$ under suitable framework, we cannot determine the asymptotic phases $\{\theta_\alpha^\infty\}_{\alpha=1}^N$ explicitly as the limit $T^\infty$ of temperature field. The reason here is of course the lack of conservation law for $\Theta$, which is one of the main difference between TK model and Kuramoto model. \\

In Kuramoto model \eqref{D-2}, the total sum $\sum_{\alpha}\phi_\alpha$ is conserved in time, while we only obtain the convergence of it in TK model by using Corollary \ref{C2.1}. Then, as the equilibrium solution is determined by the equations of relative phases
\begin{equation}\label{phaselock}
\nu_\alpha+\frac{\kappa_1}{NT^\infty}\sum_{\beta}\psi_{\alpha \beta}\sin(\phi_\beta^\infty-\phi_\alpha^\infty)=0,\quad \alpha=1,\cdots,N,
\end{equation}
one can explicitly determine $\{\phi_\alpha^\infty\}_{\alpha=1}^N$ by solving \eqref{phaselock} and with
\[\sum_{\alpha}\theta_\alpha^{in}=\sum_{\alpha}\phi_\alpha^\infty. \]
Now, since every equilibrium solution $(\Theta^\infty,T^\infty)$ of TK model satisfies the equations
\begin{equation*}
T^\infty+\frac{\eta^2}{2T_*^2}{T^\infty}^2=\frac{1}{N}\sum_{\alpha} \left({T^{in}_\alpha}+\frac{\eta^2}{2T_*^2}{T^{in}_\alpha}^2\right), \quad  \nu_\alpha +\frac{\kappa_1}{NT^\infty}\sum_{\beta}\psi_{\alpha \beta}\sin(\theta_{\beta}^\infty-\theta_{\alpha}^\infty)=0,
\end{equation*}
it is clear that the limit of average phase $\theta_c^\infty:=\frac{1}{N}\sum_{\alpha}\theta_\alpha^\infty$ for the TK model will also determine the limit configuration $\Theta^\infty$ as in Kuramoto model. Hence, we here provide an upper bound estimate of average phase limit $\theta_c^\infty$ for later use.  

\begin{lemma}\label{L4.1}
    Let $\{(\theta_\alpha,T_\alpha)\}_{\alpha=1}^N$ be a solution to \eqref{D-1} with the initial data $\{(\theta_\alpha^{in},T_\alpha^{in})\}_{\alpha=1}^N$ satisfying 
    \[\kappa_1, \kappa_2>0,\quad \psi_{\min}>0,\quad \zeta_{\min}>0,\quad 0<T_1^{in}\leq\cdots\leq T_N^{in}<\infty. \]
    Then, the limit of average phase $\theta_c^\infty$ satisfies 
    \[\left|\theta_c^\infty-\theta_c^{in}\right|\leq \frac{\kappa_1\psi_{\max}}{\kappa_2\zeta_{\min}}\cdot \frac{{T_N^{in}}^2(T_*^2+\eta^2{T_N^{in}})(T_N^{in}-T_1^{in})}{T_*^2T^\infty T_1^{in}}, \]
    where $\theta_c^{in} :=\frac{1}{N}\sum_{\alpha}\theta_\alpha^{in}$.
\end{lemma}
\begin{proof}
    We rewrite the differential equation for $\theta_c$ as below: 
    \[\begin{aligned}
    \dot\theta_c=\frac{\kappa_1}{N^2}\sum_{\alpha, \beta}\frac{\psi_{\alpha\beta}}{T_\alpha}\sin(\theta_\beta-\theta_\alpha)=\frac{\kappa_1}{N^2}\sum_{\alpha, \beta}\psi_{\alpha\beta}\left(\frac{1}{T_\alpha}-\frac{1}{T_\infty}\right)\sin(\theta_\beta-\theta_\alpha).
    \end{aligned} \]
    Then, we use the exponential decay estimate in Proposition \ref{P2.4} to deduce
    \[\begin{aligned}
    \left|\theta_c^\infty-\theta_c^{in}\right|&\leq \int_{0}^{\infty}|\dot\theta_c(t)| dt \leq \frac{\kappa_1\psi_{\max}}{T_1T_\infty}\int_{0}^{\infty}|T_N(t)-T_1(t)| dt\\
    &\leq \frac{\kappa_1\psi_{\max}}{T_1T_\infty}\int_{0}^{\infty}|T_N^{in}-T_1^{in}|e^{-\frac{\kappa_2\zeta_{\min}T_*^2}{{T^{in}_N}^2(T_*^2+\eta^2{T^{in}_N})}t} dt\\
    &=\frac{\kappa_1\psi_{\max}}{\kappa_2\zeta_{\min}}\cdot \frac{{T_N^{in}}^2(T_*^2+\eta^2{T_N^{in}})(T_N^{in}-T_1^{in})}{T_*^2T^\infty T_1^{in}}.
    \end{aligned} \]
\end{proof}
In \cite{C-H-J-K}, authors provided $\ell^1$-stability result for the Kuramoto model \eqref{D-2} with $\psi_{\alpha\beta}\equiv1$ when the initial configuration is confined in a half circle, and coupling strength $\kappa_1$ is sufficiently large. As a corollary, they obtained the uniqueness of equilibrium state in a half circle for given initial average phase $\theta_c^{in}$. If we can prove the analogous result for generic static network $\psi_{\alpha \beta}=\psi_{\beta\alpha}>0$, the phase limit $\Theta^\infty$ can be represented by the unique solution of the Kuramoto model and the distance between the average phase limits of TK model and Kuramoto model. 

\begin{lemma}\label{L4.2}
    Let $\Phi = \{ \phi_\alpha \}_{\alpha=1}^N$ and ${\tilde \Phi} = \{ \widetilde\phi_\alpha \}_{\alpha=1}^N$ be two solutions to the Kuramoto model \eqref{D-2} with the initial data $(\phi_\alpha^{in})_{\alpha=1}^N ,(\widetilde\phi_\alpha^{in})_{\alpha=1}^N$ satisfying 
    \[\kappa_1>0,\quad \sum_{\alpha}\phi_\alpha^{in}=\sum_{\alpha}\widetilde\phi_\alpha^{in},\quad \frac{{\mathcal D}(\nu) T^\infty}{\kappa_1\psi_{\min}}<1, \quad \max\{ {\mathcal D}(\phi^{in}),{\mathcal D}(\widetilde\phi^{in})  \} <\pi-\arcsin\left(\frac{{\mathcal D}(\nu) T^\infty}{\kappa_1\psi_{\min}}\right).
    \]
    Then, we have 
    \begin{equation}\label{stability}
   \| \Phi(t) - {\tilde \Phi}(t) \|_{1} \lesssim e^{-\frac{\kappa_1\psi_{\min}}{T^\infty}\cdot \frac{\sin2D^\infty}{2D^\infty}t},
    \end{equation}
    for every $D^\infty\in \left(\arcsin\left(\frac{{\mathcal D}(\nu) T^\infty}{\kappa_1\psi_{\min}}\right),\frac{\pi}{2}\right)$. 
\end{lemma}
\begin{proof}
    Although the proof is indeed analogous to \cite{C-H-J-K}, we here present a proof for reader's convenience. \\
    
    \noindent First, it follows from the analyticity of $\phi_\alpha$ and $\widetilde\phi_\alpha$ that the set $\left\{t\geq 0: \phi_\alpha(t)-\widetilde\phi_\alpha(t)=0 \right\}$ is either discrete or whole time $[0,\infty)$. Therefore, we can find an increasing sequence of times $\{\tau_i\}_{i\geq 0}$ such that for each $i\geq 1$ and $\alpha=1,\cdots,N$, the sign of $\phi_\alpha-\widetilde\phi_\alpha$ is unchanged in $(\tau_{i-1},\tau_i)$:
    \[\textrm{sgn}\left(\phi_\alpha(t_1)-\widetilde\phi_\alpha(t_1)\right)=\textrm{sgn}\left(\phi_\alpha(t_2)-\widetilde\phi_\alpha(t_2)\right),\quad \forall ~ t_1, t_2\in(\tau_{i-1},\tau_i), \]
    where $\textrm{sgn}:\mathbb{R}\to\mathbb{R}$ is a function defined by
    \[\textrm{sgn}(x)=\begin{cases}
    1 & x>0,\\
    0 & x=0,\\
    -1 & x<0.
    \end{cases} \]
    Then, the $\ell^1$-distance function $\sum_{\alpha} |\phi_\alpha-\widetilde\phi_\alpha|$ is differentiable in each $(\tau_{i-1},\tau_i)$, and the derivative can be written as 
    \[\begin{aligned}
    \frac{d}{dt} \left(\sum_{\alpha} |\phi_\alpha-\widetilde\phi_\alpha|\right)= \frac{\kappa_1}{NT^\infty}\sum_{\alpha, \beta}\psi_{\alpha \beta}\textrm{sgn}(\phi_{\alpha}-\widetilde\phi_\alpha)\left(\sin(\phi_\beta-\phi_\alpha)-\sin(\widetilde\phi_\beta-\widetilde\phi_\alpha) \right).
    \end{aligned} \]
    Now, we denote $z_\alpha:=\phi_\alpha-\widetilde\phi_\alpha$ for simplicity and rewrite the above equation as below: 
    \begin{align}\label{distance}
    \begin{aligned}
    &\frac{d}{dt}\left(\sum_{\alpha}|z_\alpha|\right) \\
    & \hspace{0.5cm} =\frac{\kappa_1}{NT^\infty}\sum_{\alpha, \beta}\psi_{\alpha \beta}\textrm{sgn}(z_\alpha)\left(\sin(\phi_\beta-\phi_\alpha)-\sin(\widetilde\phi_\beta-\widetilde\phi_\alpha) \right)\\
    & \hspace{0.5cm}  =\frac{2\kappa_1}{NT^\infty}\sum_{\alpha, \beta}\psi_{\alpha \beta}\textrm{sgn}(z_\alpha)\cos\left(\frac{\phi_\beta-\phi_\alpha}{2}+\frac{\widetilde\phi_\beta-\widetilde\phi_\alpha}{2} \right)\sin\left(\frac{z_\beta-z_\alpha}{2}\right)\\
    & \hspace{0.5cm} = \frac{\kappa_1}{NT^\infty}\sum_{\alpha, \beta}\psi_{\alpha \beta}\left(\textrm{sgn}(z_\alpha)-\textrm{sgn}(z_\beta)\right)\sin\left(\frac{z_\beta-z_\alpha}{2}\right)\cos\left(\frac{\phi_\beta-\phi_\alpha}{2}+\frac{\widetilde\phi_\beta-\widetilde\phi_\alpha}{2} \right).
    \end{aligned}
    \end{align} 
    On the other hand, since we can apply Proposition \ref{P3.2} to the Kuramoto model \eqref{D-2} as a special case $T_1^{in}=\cdots=T_N^{in}=T^\infty$ for \eqref{D-1}, we know that the phase diameters $\mathcal{D}(\phi(t))$ and $\mathcal{D}(\widetilde\phi(t))$ are uniformly bounded by $\pi$ and become smaller than $\frac{\pi}{2}$ after sufficiently long time.  More precisely, for every $D^\infty\in \left(\arcsin\left(\frac{{\mathcal D}(\nu) T^\infty}{\kappa_1\psi_{\min}}\right),\frac{\pi}{2}\right)$ there exists a time $T_{D^\infty}$ such that 
    \[\mathcal{D}(\phi(t)),\mathcal{D}(\widetilde\phi(t))<D_\infty, \quad \forall~t>T_{D^\infty}. \]
    Therefore, for $t>T_{D^\infty}$, we have 
    \begin{equation}\label{sign}
    \left|\frac{\phi_\beta-\phi_\alpha}{2}+\frac{\widetilde\phi_\beta-\widetilde\phi_\alpha}{2}\right|<D_\infty,\quad |z_\beta-z_\alpha|<2D^\infty<\pi,
    \end{equation}
    and deduce a following Grönwall's inequality from \eqref{distance}:
    \begin{align}\label{Grnwll}
    \begin{aligned}
    \frac{d}{dt}\left(\sum_{\alpha}|z_\alpha|\right)&\leq -\frac{\kappa_1\cos D^\infty}{NT^\infty}\sum_{\alpha, \beta} \psi_{\alpha \beta}\left(\textrm{sgn}(z_\alpha)-\textrm{sgn}(z_\beta)\right)\sin\left(\frac{z_\alpha-z_\beta}{2}\right)\\
    &\leq -\frac{\kappa_1\psi_{\min}\cos D^\infty}{NT^\infty}\sum_{\alpha, \beta} \left(\textrm{sgn}(z_\alpha)-\textrm{sgn}(z_\beta)\right)\sin\left(\frac{z_\alpha-z_\beta}{2}\right)\\
    &\leq -\frac{\kappa_1\psi_{\min}\cos D^\infty\sin D^\infty}{2NT^\infty D^\infty}\sum_{\alpha, \beta} \left(\textrm{sgn}(z_\alpha)-\textrm{sgn}(z_\beta)\right)\left({z_\alpha-z_\beta}\right)\\
    &= -\frac{\kappa_1\psi_{\min}\cos D^\infty\sin D^\infty}{T^\infty D^\infty}\left(\sum_{\alpha}|z_\alpha|\right).
    \end{aligned}
    \end{align}
    Here, we used 
    \[\left(\textrm{sgn}(z_\alpha)-\textrm{sgn}(z_\beta)\right)\sin\left(\frac{z_\alpha-z_\beta}{2}\right)\geq \left(\textrm{sgn}(z_\alpha)-\textrm{sgn}(z_\beta)\right)\left(\frac{z_\alpha-z_\beta}{2}\right)\cdot\frac{\sin D^\infty}{D^\infty}\geq 0,\]
    and 
    \[\sum_{\alpha, \beta} \left(\textrm{sgn}(z_\alpha)-\textrm{sgn}(z_\beta)\right)\left({z_\alpha-z_\beta}\right)=2N\sum_{\alpha}|z_\alpha|, \]
    which are derived from $\eqref{sign}_2$ and
    \[  \sum_{\alpha} z_\alpha=\sum_{\alpha} \phi_\alpha-\sum_{\alpha} \widetilde\phi_\alpha=0, \]
    respectively.\\
    
    Finally, we use the inequality \eqref{Grnwll} in each time interval $(\tau_{i-1},\tau_i)$ and the continuity of $z_\alpha$ to conclude the desired stability estimate \eqref{stability}.
\end{proof}

\begin{remark}
    Suppose that there are two roots $\phi^\infty$ and $\widetilde{\phi}^\infty$ of $\eqref{phaselock}$ satisfying the condition 
    \begin{equation}\label{cond}
    \frac{{\mathcal D}(\nu) T^\infty}{\kappa_1\psi_{\min}}<1,\quad \sum_{\alpha}\phi_\alpha^\infty=0, \quad {\mathcal D}(\phi^{\infty}) <\pi-\arcsin\left(\frac{{\mathcal D}(\nu) T^\infty}{\kappa_1\psi_{\min}}\right).
    \end{equation}
    Then, the $\ell^1$-distance $\sum_{\alpha} |\phi_\alpha^\infty-\widetilde\phi_\alpha^\infty|$ must be a constant, but also decays exponentially from Lemma \ref{L4.2}. This indicates that the $\ell^1$-distance is indeed $0$ and there are only one solution for \eqref{cond}.
\end{remark}
Now, we combine Lemma \ref{L4.1} and Lemma \ref{L4.2} to describe the limit state of TK model \eqref{D-1}.
\begin{theorem}\label{T4.1}
   Suppose that system parameters $\kappa_1, \kappa_2, \psi_{\alpha\beta}, \zeta_{\alpha \beta}, \nu$ and initial data $(\theta_\alpha^{in},T_\alpha^{in})_{\alpha=1}^N$ satisfy
    \eqref{cond1}--\eqref{cond2}, and let $\{ (\theta_{\alpha},T_\alpha) \}_{\alpha=1}^N$ be a solution to \eqref{D-1}. Then, there exists a constant number $z\in\mathbb{R}$ such that 
    \[\begin{aligned}
    \lim_{t \to \infty}\theta_{\alpha}(t)=\lim_{t \to \infty}\phi_\alpha(t)+z,\quad |z|\leq \frac{\kappa_1\psi_{\max}}{\kappa_2\zeta_{\min}}\cdot \frac{{T_N^{in}}^2(T_*^2+\eta^2{T_N^{in}})(T_N^{in}-T_1^{in})}{T_*^2T^\infty T_1^{in}},\quad \alpha=1,\cdots,N,
    \end{aligned} \]
    where $\{ \phi_{\alpha} \}$ is a solution to the Kuramoto model \eqref{D-2} with same initial phase configuration $(\theta_\alpha^{in})$.
\end{theorem}
\begin{proof}
	We just sketch how can conclude Theorem \ref{T4.1} from Lamma \ref{L4.1} and Lemma \ref{L4.2}. First, recall from Theorem \ref{T3.2} that the asymptotic limit $(\Theta^\infty,T^\infty)$ of $\{ (\theta_{\alpha},T_\alpha) \}_{\alpha=1}^N$ exists and satisfy \eqref{equili}. On the other hand, by using Lemma \ref{L4.2}, one can deduce that the solution of $\eqref{equili}_2$ for given phase sum $\sum_{\alpha = 1}^N\phi_\alpha^\infty$ is unique, which is also an asymptotic equilibrium of Kuramoto model \eqref{D-2}. Therefore, the asymptotic state of \eqref{D-1} can be obtained by a phase shift of the unique asymptotic state of \eqref{D-2}, and the amount of phase shift can be evaluated by Lemma \ref{L4.1}.
\end{proof}
\section{Conclusion} \label{sec:5}
\setcounter{equation}{0}
In this paper, we proposed a new Kuramoto type phase-temperature coupled model governing the dynamic evolution of Kuramoto oscillators under the effect of temperature field. In literature, the Kuramoto model has been extensively studied in the absence of temperature field.  As aforementioned in Introduction, the effect of temperature in oscillator's phase evolution has been discussed in biology literature.  Our proposed thermodynamic Kuramoto model was derived from the thermodynamic Cucker-Smale model satisfying the entropy principle. Hence, naturally it inherits entropy principle from the TCS model. For emergent dynamics, we provide several sufficient frameworks in terms of system parameters and initial phase-temperature configurations. As far as authors know, our proposed model serves the first mathematical modeling for Kuramoto oscillators undergoing the effect of their own temperature, and more specific analysis on asymptotic behaviors of the TK model will be addressed in subsequent studies.

\bibliographystyle{amsplain}

\end{document}